\documentclass[a4paper,11pt,reqno]{amsart}
\usepackage{amsmath}
\usepackage{amsthm}
\usepackage{amssymb}
\usepackage{amscd}
\usepackage{amsfonts}
\usepackage{amsbsy}
\usepackage{enumerate}
\usepackage{graphicx}
\usepackage[dvips]{psfrag}
\usepackage[latin1]{inputenc}
\usepackage{color}

\newcommand{\R}{\ensuremath{\mathbb{R}}}
\newcommand{\C}{\ensuremath{\mathbb{C}}}
\newcommand{\CP}{\mathbb{C}\mathbb{P}}

\newcommand{\pd}[2]{\frac{\partial #1}{\partial #2}}

\newcommand{\cqb}{\hbox{}\nobreak\hfill$\Box$}
\newcommand{\oo}{\mathrm{ord}}
\newcommand{\Sing}{\mathrm{Sing}}

\newcounter{teoremaganso}

\newtheorem{theorem}{Theorem}
\newtheorem{proposition}[theorem]{Proposition}
\newtheorem{corollary}[theorem]{Corollary}
\newtheorem{lemma}[theorem]{Lemma}
\newtheorem {exam} {Example}
\newtheorem {remark} {Remark}

\newenvironment{example}{\begin{exam}\rm}{\cqb\end{exam}}

\advance\textwidth by +1in
\advance\textheight by +1in
\advance\oddsidemargin by -0.5in
\advance\evensidemargin by -0.5in
\advance\topmargin by -0.5in

\begin{document}



\title[Quadratic systems with algebraic limit cycles via Cremona maps]{Quadratic planar differential systems\\ with algebraic limit cycles\\ via quadratic plane Cremona maps}

\author[M. Alberich-Carrami\~nana, A. Ferragut, J. Llibre]{Maria Alberich-Carrami\~nana, Antoni Ferragut, Jaume Llibre}

\address{M. Alberich-Carrami\~nana: Institut de Rob\`otica i Inform\`atica Industrial (IRI, CSIC-UPC) and Departament de Matem\`atiques, Universitat Polit\`ecnica de Cata\-lunya, Av. Diagonal, 647, 08028 Barcelona, Spain}

\email{maria.alberich@upc.edu}

\address{A. Ferragut: Universidad Internacional de la Rioja. Avenida de la Paz, 137, 26006 Logro\~no, Spain}

\email{toni.ferragut@unir.net}

\address{J. Llibre: Departament de Matem\`{a}tiques, Universitat Aut\`{o}noma de Barcelona, 08193 Bella\-terra, Barcelona, Catalonia-Spain}

\email{jllibre@mat.uab.cat}


\date{\today}


\begin{abstract}
In this paper we show how we can transform quadratic systems into new quadratic systems after some kind of birational transformations, the quadratic plane Cremona maps. We  afterwards apply these transformations to the families of quadratic differential systems having an algebraic limit cycle. As a consequence, we  provide a new family of quadratic systems having an algebraic limit cycle of degree 5. Moreover we show how the known families of quadratic differential systems having an algebraic limit cycle of degree greater than four are obtained using these transformations.  We also provide the phase portraits on the Poincar\'e disk of all the families of quadratic differential systems having algebraic limit cycles.
\end{abstract}


\maketitle

\section{Introduction}

We consider the quadratic planar  differential system
\begin{equation}\label{e1}
\dot x=p(x,y),\quad \dot y=q(x,y),
\end{equation}
where $p(x,y)$ and $q(x,y)$ are real coprime polynomials of degree two. Let $f \in {\R}[x,y]$. We say that $f=f(x,y)=0$ is an {\it invariant algebraic curve} of system \eqref{e1} if it satisfies
\begin{equation}
p(x,y)\pd fx(x,y)+q(x,y)\pd fy(x,y)=k(x,y)f(x,y),
\end{equation}
for some $k(x,y)$ polynomial of degree at most $1$ called the {\it cofactor} of $f(x,y)=0$.  If $f \in {\R}[x,y]$ has degree $n$, it is irreducible in ${\R}[x,y]$ and $f=0$ is an invariant algebraic curve, then we say that $f=0$ is an  {\it irreducible invariant algebraic curve of degree n}.

\smallskip

A {\it limit cycle} of system \eqref{e1} is an isolated periodic solution in the set of all periodic solutions of the system. If a limit cycle is contained into the set of points of an invariant algebraic curve, then it is called an {\it algebraic limit cycle}. We say that an algebraic limit cycle has {\it degree} $n$ if it is contained into the set of points of an irreducible invariant algebraic curve of degree $n$.

\smallskip

One of the most interesting questions on limit cycles was proposed by Hilbert \cite {H} in 1900 in the second part of $16^{th}$ Hilbert's Problem: {\it Compute $H(m)$ such that the number of limit cycles of any polynomial differential system of degree $m$ is less than or equal to $H(m)$}.

\smallskip

Hilbert's Problem remains unsolved even for $m=2$. It is known that a quadratic system with an invariant straight line has at most one limit cycle (see \cite{Co} or \cite{CL}).

\smallskip

The paper is structured as follows. In section \ref{S.Known} we provide the known families of planar quadratic differential systems having an algebraic limit cycle. We also provide their phase portrait in the Poincar\'e disk, which was never done before. In section \ref{S.PVF+Cre}  we first introduce the plane Cremona maps, in particular the quadratic ones. Afterwards we state and prove some results connecting local and global behavior that allow us to know {\em a priori} whether a Cremona transformation can be applied to obtain a new quadratic system, according to the local behavior of the base points. The degree of the transformed algebraic curve is also computed. To finish this section we study the particular case of quadratic Cremona maps applied to quadratic differential systems, which is the main aim of this work. The rest of main results are presented in section \ref{S.results}. Theorem \ref{T.C2a} shows under which conditions we can obtain a quadratic system after applying a quadratic Cremona map to a quadratic differential system. Afterwards, Theorem \ref{T.New5} provides a new quadratic differential system having an algebraic limit cycle of degree 5. These results are proved in sections \ref{P.TC2a} and \ref{S:new5}, respectively. Finally section \ref{S.rel} shows the relations among  the known quadratic systems with algebraic limit cycles via quadratic Cremona maps.

\section{The known families of planar quadratic differential systems having an algebraic limit cycle}\label{S.Known}

There is exactly one family having an algebraic limit cycle of degree 2, found by Qin in 1958 (see \cite{Q}).  Evdokimenco from 1970 to 1979 proved that there are no quadratic systems having an algebraic limit cycle of degree 3 (see  \cite{E1,E2,E3}), see Theorem 11 of \cite{CLM} for a short proof. There are four families having an algebraic limit cycle  of degree 4: the first one was found by Yablonskii in 1966 (see \cite{Y}); the second one was found by Filipstov in 1973 (see \cite{F}); Chavarriga found the third one and Chavarriga, Llibre and Sorolla found the fourth one, they both were published in 2004 in \cite{CLS2004}. It was also proved in \cite{CLS2004} that there are no other families having algebraic limit cycles of degree 4. Finally, up to now there were only one known family having an algebraic limit cycle of degree 5 and one known family having an algebraic limit cycle of degree 6, both of them found in 2005 (see \cite{CLS}). Both families are found after a birational transformation of the family due to Chavarriga, Llibre and Sorolla of \cite{CLS2004}. Moreover, in that paper a birational transformation relates Yablonskii's family and Qin's family. Since that paper no other families of quadratic systems having an algebraic limit cycle have been found.

\smallskip

Qin Yuan-Sh\"un summarizes in 1958 (see \cite{Q}) the quadratic systems having an algebraic limit cycle of degree 2 and he proves the uniqueness of this limit cycle:

\begin{proposition}[Qin limit cycle]\label{P.Qin}
If a quadratic system has an algebraic limit cycle of degree 2, then after an affine change of variables and time, the limit cycle becomes the circle $x^2+y^2-1=0$. Moreover, it is the unique limit cycle of the quadratic differential system, which can be written as
\begin{equation}\label{QinQS}
\begin{split}
\dot{x}&=-y(ax+by+c)-(x^2+y^2-1)\ ,\\
\dot{y}&=x(ax+by+c)\ ,
\end{split}
\end{equation}
with $a\neq 0$, $c^2+4(b+1)>0$ and $a^2+b^2<c^2$.
\end{proposition}

The case of the limit cycles of degree 3 was studied later on. Using three papers Evdokimenco proved from 1970 to 1979 that there are no quadratic systems having limit cycles of degree 3 (see \cite{E1,E2,E3}). A simpler proof can be found in \cite{CLM}.

\smallskip

Yablonskii \cite{Y} found the first family of quadratic differential systems having an algebraic limit cycle of degree $4$ in 1966:

\begin{proposition}[Yablonskii limit cycle]\label{P.Yab}
The quadratic differential system
\begin{equation}\label{YabQS}
\begin{split}
\dot{x}&=-4abcx-(a+b)y+3(a+b)cx^2+4xy \ , \\
\dot{y}&=(a+b)abx-4abcy+(4abc^2-{\frac 32}(a+b)^2+4ab)x^2+8(a+b)cxy+8y^2 \ , \\
\end{split}
\end{equation}
with $abc\neq 0$, $a\neq b$, $ab>0$ and $4c^2(a-b)^2+(3a-b)(a-3b)<0$, has the irreducible invariant algebraic curve
\[
(y+cx^2)^2+x^2(x-a)(x-b)=0
\]
of degree 4 having two components: an oval (the algebraic limit cycle) and an isolated singular point.
\end{proposition}

In 1973 a new family of algebraic limit cycles of degree 4 was found by Filipstov \cite{F}:

\begin{proposition}[Filipstov limit cycle]\label{P.Fil}
The quadratic differential system
\begin{equation}\label{FilQS}
\begin{split}
\dot{x}&=6(1+a)x+2y-6(2+a)x^2+12xy , \\
\dot{y}&=15(1+a)y+3a(1+a)x^2-2(9+5a)xy+16y^2 ,
\end{split}
\end{equation}
with $0<a<3/13$, has the irreducible invariant algebraic curve
\[
3(1+a)(ax^2+y)^2+2y^2(2y-3(1+a)x)=0
\]
of degree 4 having two components: one is an oval and the other one is homeomorphic to a straight line. This second component contains three singular points of the system.
\end{proposition}

The third algebraic limit cycle of degree four was found by Chavarriga in 1999, although the result  was first published in 2004 in \cite{CLS2004}:

\begin{proposition}[Chavarriga limit cycle]\label{P.Cha}
The quadratic differential system
\begin{equation}\label{ChaQS}
\begin{split}
\dot{x}&=5x+6x^2+4(1+a)xy+ay^2 , \\
\dot{y}&=x+2y+4xy+(2+3a)y^2 ,
\end{split}
\end{equation}
with $(-71+17\sqrt{17})/ 32<a<0$, has the irreducible invariant algebraic curve
\[
x^2+x^3+x^2y+2axy^2+2axy^3+a^2y^4=0
\]
of degree 4. It has three components; one of them is an oval and each one of the others is homeomorphic to a straight line. Each one of these last two components contains one singular point of the system.
\end{proposition}

Finally, Chavarriga, Llibre and Sorolla \cite{CLS2004} in 2004 found the fourth one:

\begin{proposition}[Chavarriga, Llibre and Sorolla limit cycle]\label{P.CLS}
The quadratic differential system
\begin{equation}\label{CLS4QS}
\begin{split}
\dot{x}&=2(1+2x-2ax^2+6xy), \\
\dot{y}&=8-3a-14ax-2axy-8y^2,
\end{split}
\end{equation}
with $0<a<1/4$, possesses the irreducible invariant algebraic curve
\[
\frac 14+x-x^2+ax^3+xy+x^2y^2=0
\]
of degree 4 having three components; one of them is an oval and each of the others is homeomorphic to a straight line. One of these last two components contains two singular points of the system, the other does not contain any singular point.
\end{proposition}
In what follows we denote this system by {\rm CLS}. We have corrected a mistake in \cite{CGL} concerning the singular points on the components of the algebraic curve of \eqref{CLS4QS}.

\smallskip

It was proved in \cite{CLS2004} that there are no other families of quadratic systems having algebraic limit cycles of degree 4. That is, after an affine change of variables and time, the unique quadratic systems having an algebraic limit cycles of degree 4 are the previous ones.

\smallskip

Concerning families of quadratic systems having algebraic limit cycles of degree greater than 4, up to now there were only one known family having an algebraic limit cycle of degree 5 and one known family having an algebraic limit cycle of degree 6. Both on them were presented by Christopher, Llibre and \'Swirszcz in 2005 (see \cite{CLS}):

\begin{proposition}[Christopher, Llibre and \'Swirszcz limit cycle of degree 5]\label{P.CLS5}
The quadratic differential system
\begin{equation}\label{CLS5QS}
\begin{split}
\dot x&=28x+2 (16-\alpha^2) (\alpha+12) x^2+6 (3 \alpha-4) xy-\frac{12 }{\alpha+4}y^2,\\
\dot y&=2 (16-\alpha^2)x+8 y+(16-\alpha^2) (\alpha+12)  xy+2 (5 \alpha-12) y^2,
\end{split}
\end{equation}
where $\alpha\in(3\sqrt 7/2,4)$, has an algebraic limit cycle contained into the algebraic curve of degree 5
\[
\begin{split}
f(x,y)=&\,x^2+(16-\alpha^2)  x^3+(\alpha-2) x^2y-\frac{2}{\alpha+4}x y^2-\frac{1}{4}   (4-\alpha) (\alpha+12)x^2 y^2\\
&+\frac{8-\alpha}{\alpha+4}x y^3+\frac{1}{(\alpha+4)^2}y^4+ \frac{\alpha+12 }{\alpha+4}xy^4-\frac{6 }{(\alpha+4)^2}y^5.
\end{split}
\]
The cofactor of this curve is
\[
k(x,y)=56 + 6 (16- \alpha^2) (\alpha+12) x + 4(13 \alpha-24 ) y.
\]
The curve has two components; one of them is an oval and  the other is homeomorphic to a straight line. This last component contains two singular points of the system.
\end{proposition}

We shall denote   system \eqref{CLS5QS} by {\rm CLS5}.

\begin{proposition}[Christopher, Llibre and \'Swirszcz limit cycle of degree 6]\label{P.CLS6}
The quadratic differential system
\begin{equation}\label{CLS6QS}
\begin{split}
\dot x&=28\beta ( \beta - 30 )  x + y + 1 6 8 \beta^2 x^2 + 3 x y,\\
\dot y&=224\beta^2(\beta-30)^2 x+516\beta(\beta-30) y+1344\beta^3(\beta-30)  x^2 +24\beta (17\beta-6)xy+6y^2,
\end{split}
\end{equation}
where $\beta\in(3/2,2)$, has an algebraic limit cycle contained into the algebraic curve of degree 6
\[
\begin{split}
f(x,y)=& 48\beta^3(\beta - 30)^4x^2+ 24\beta^2(\beta - 30)^3xy+ 3\beta(\beta - 30)^2  y^2\\
& + 64\beta^3(\beta - 30)^3(9\beta - 4)x^3+ 24\beta^2(\beta - 30)^2(9\beta - 4)x^2y \\
& + 18\beta(\beta - 30)( \beta-2) xy^2 -7y^3 + 576\beta^3(\beta - 30)^2( \beta-2)^2x^4 \\
&+ 144\beta^2(\beta - 30)(\beta - 2)^2x^3y + 27\beta(\beta - 2)^2 x^2y^2\\
&- 3456\beta^3(\beta - 30)(\beta-2)^2( 2\beta+3)x^5- 432\beta^2(\beta - 2)^2( 2\beta+3)x^4y  \\
&+ 3456\beta^3(\beta - 2)^2( \beta+12)( 2\beta+3)x^6.
\end{split}
\]
The cofactor of this curve is
\[
k(x,y)=168\beta (\beta-30)   + 1008 \beta^2 x + 18 y.
\]
The curve has two components; one of them is an oval and  the other is homeomorphic to a straight line. This last component contains three singular points of the system.
\end{proposition}

We shall denote  this system by {\rm CLS6}.


\smallskip

Figure \ref{F.PP} shows the phase portraits of all these seven families of quadratic systems in the Poincar\'e disk. We note that Qin's system has two topologically non-equivalent phase portraits depending on the parameters.

\begin{figure}[th]
\centering
\begin{tabular}{ccc}
\includegraphics[width=3.5cm]{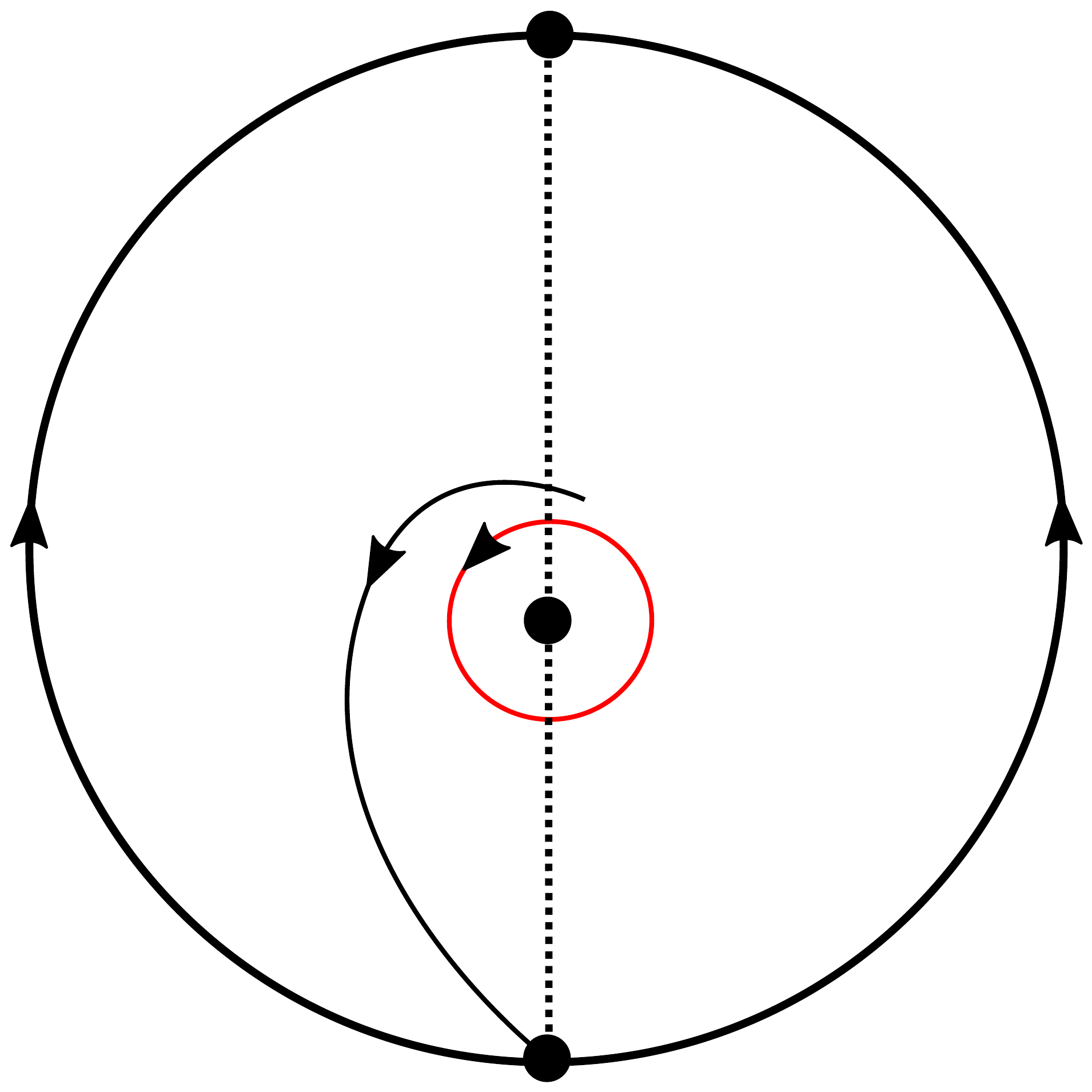}&\includegraphics[width=3.5cm]{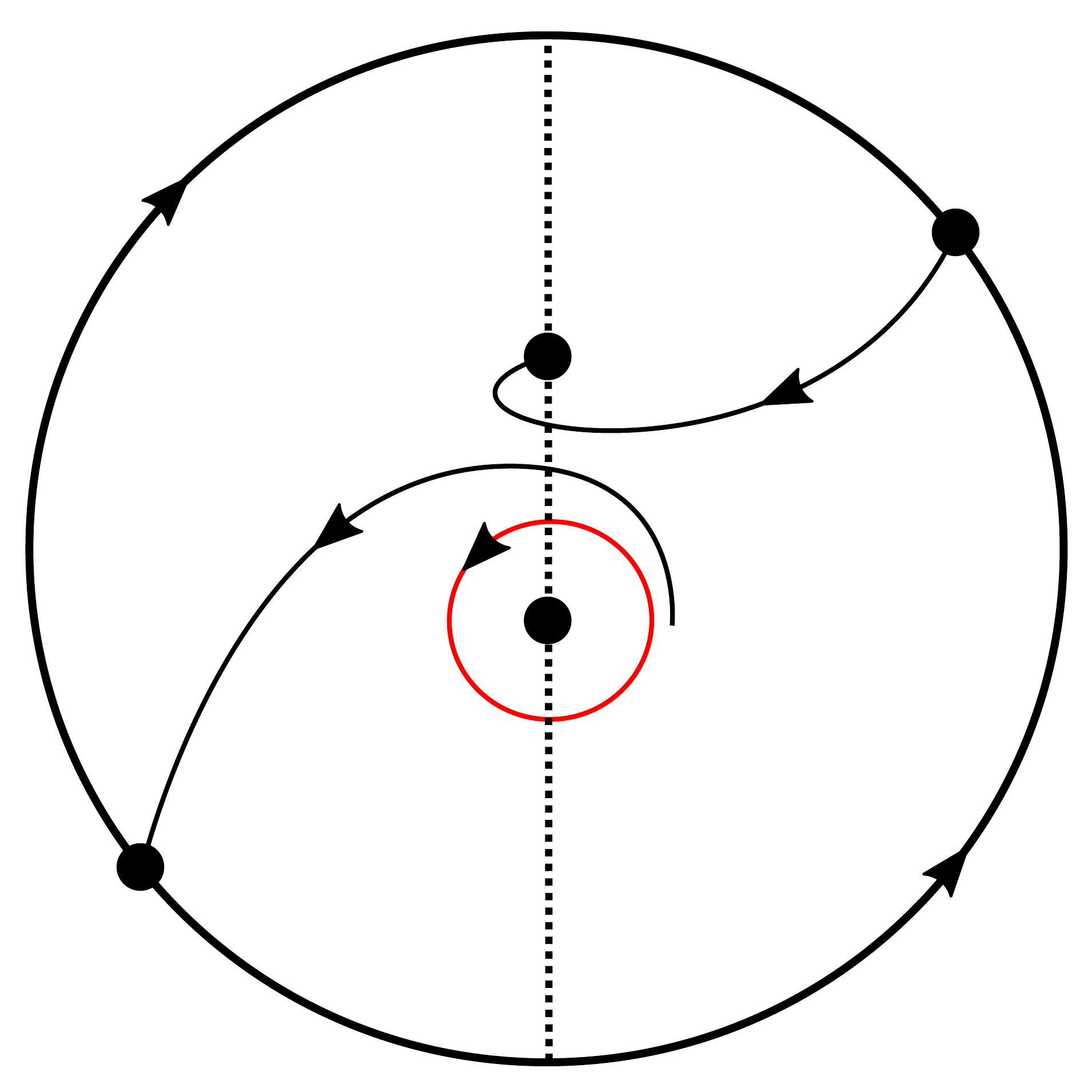}&\includegraphics[width=3.5cm]{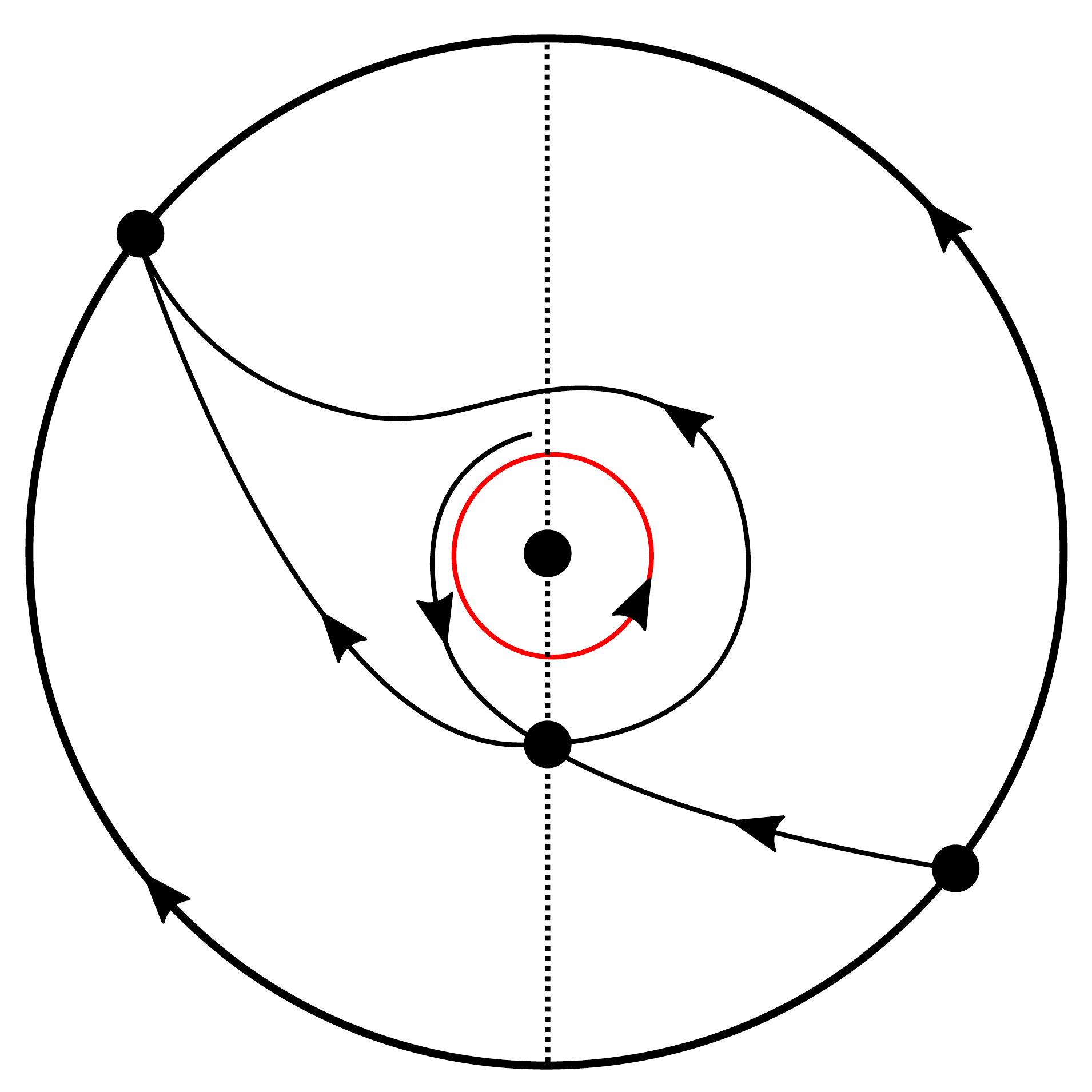}\cr
Qin with $b=-1$&Qin with $b<-1$&Qin with $b>-1$\cr\cr
\includegraphics[width=3.5cm]{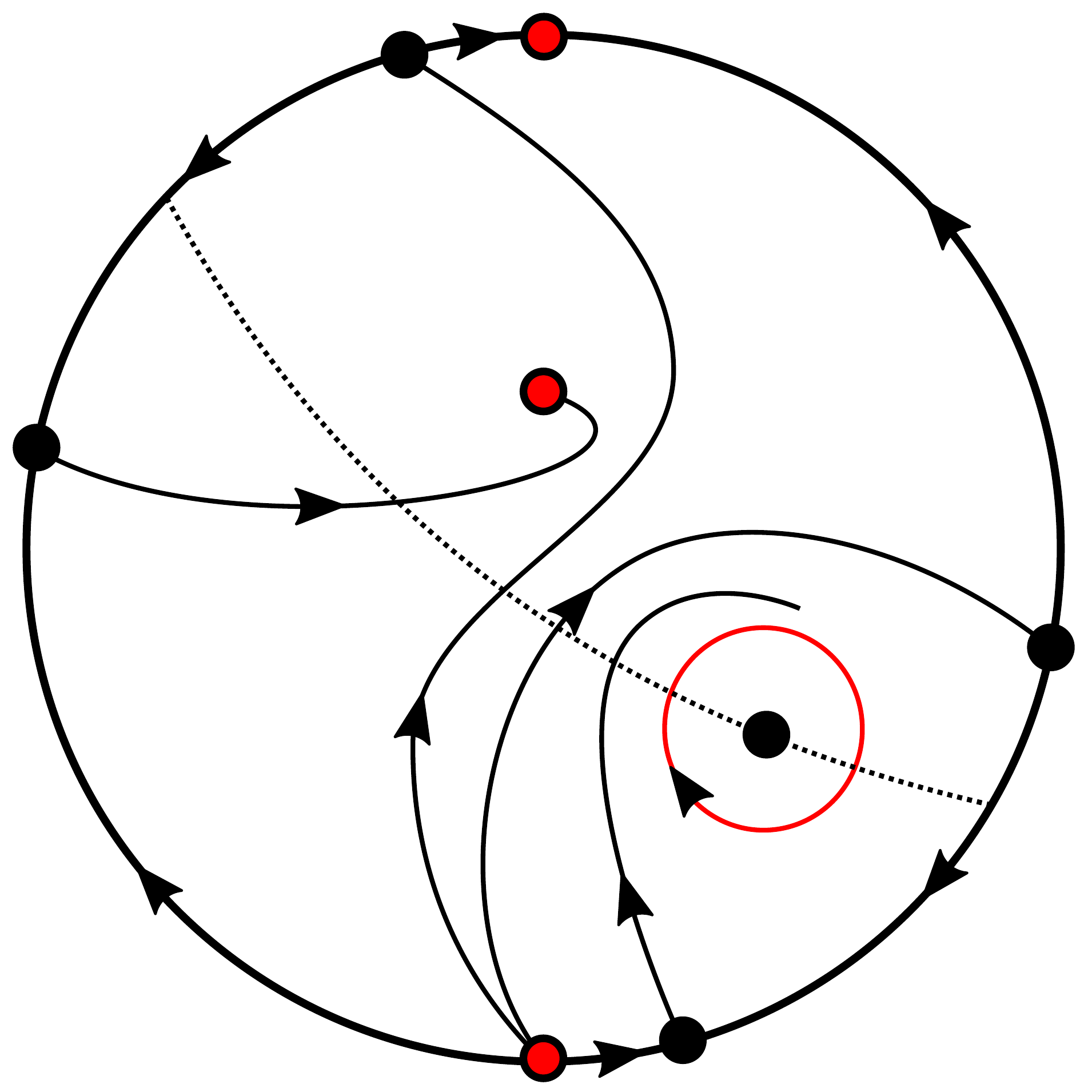}&\includegraphics[width=3.5cm]{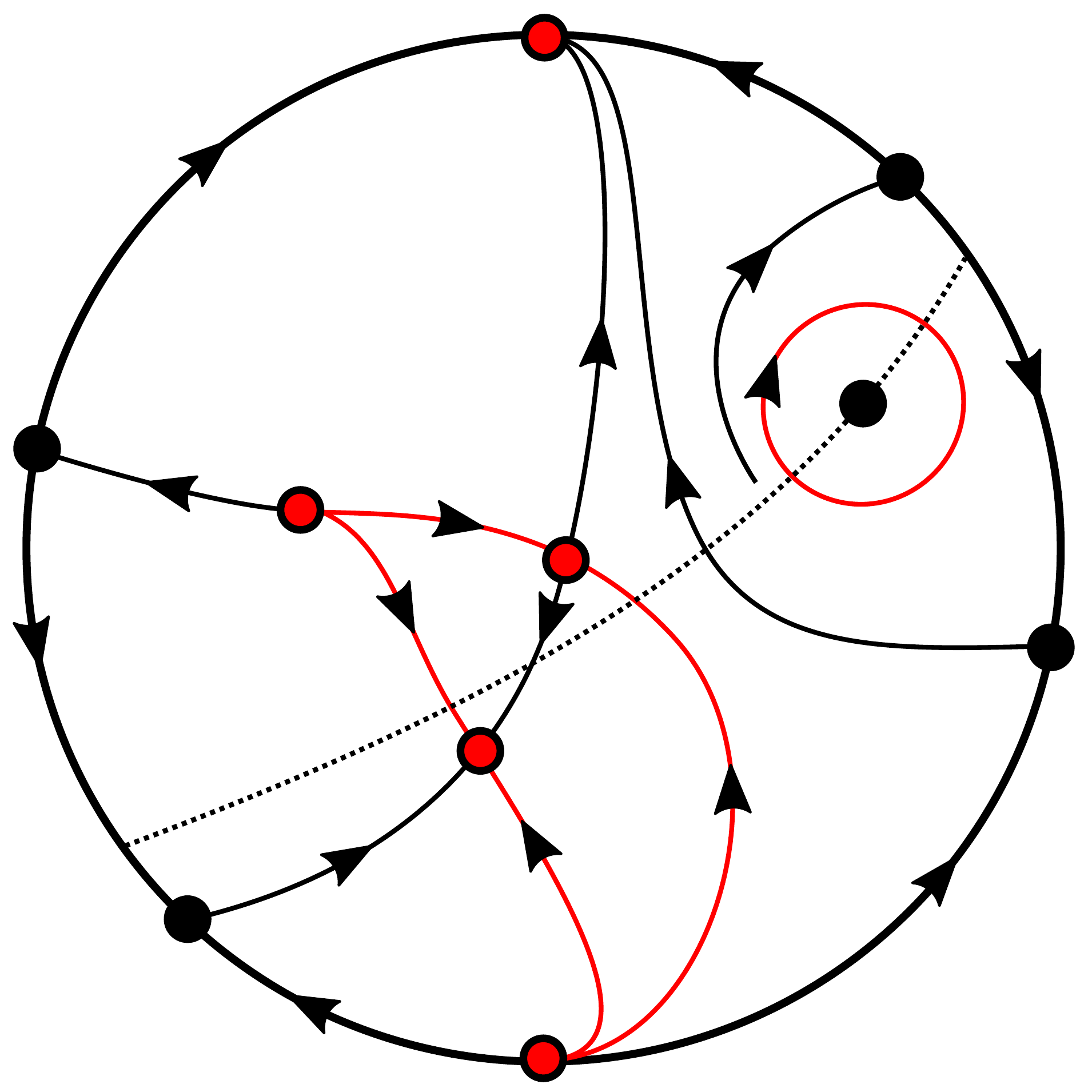}&\includegraphics[width=3.5cm]{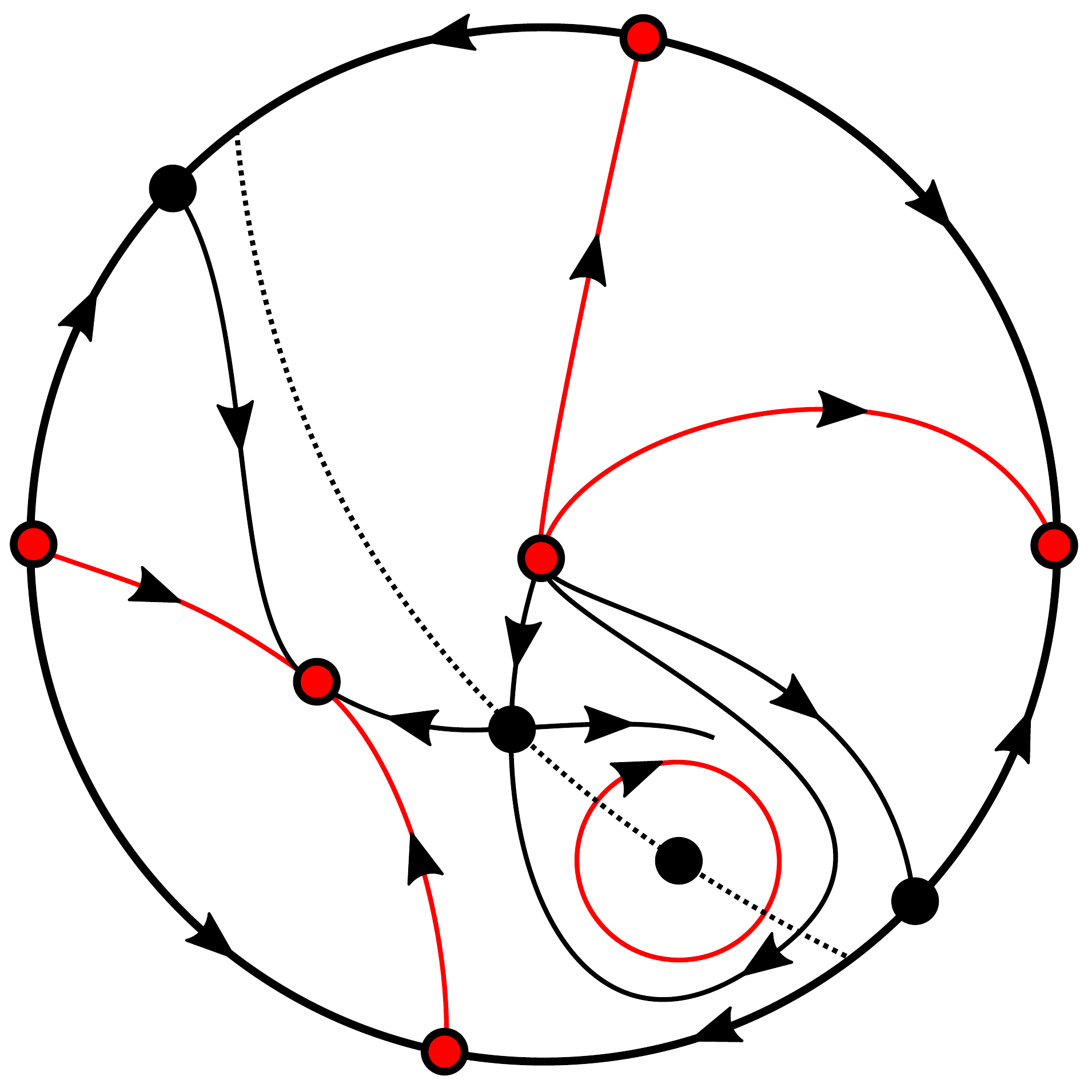}\cr
Yablonskii&Filipstov&Chavarriga\cr\cr
\includegraphics[width=3.5cm]{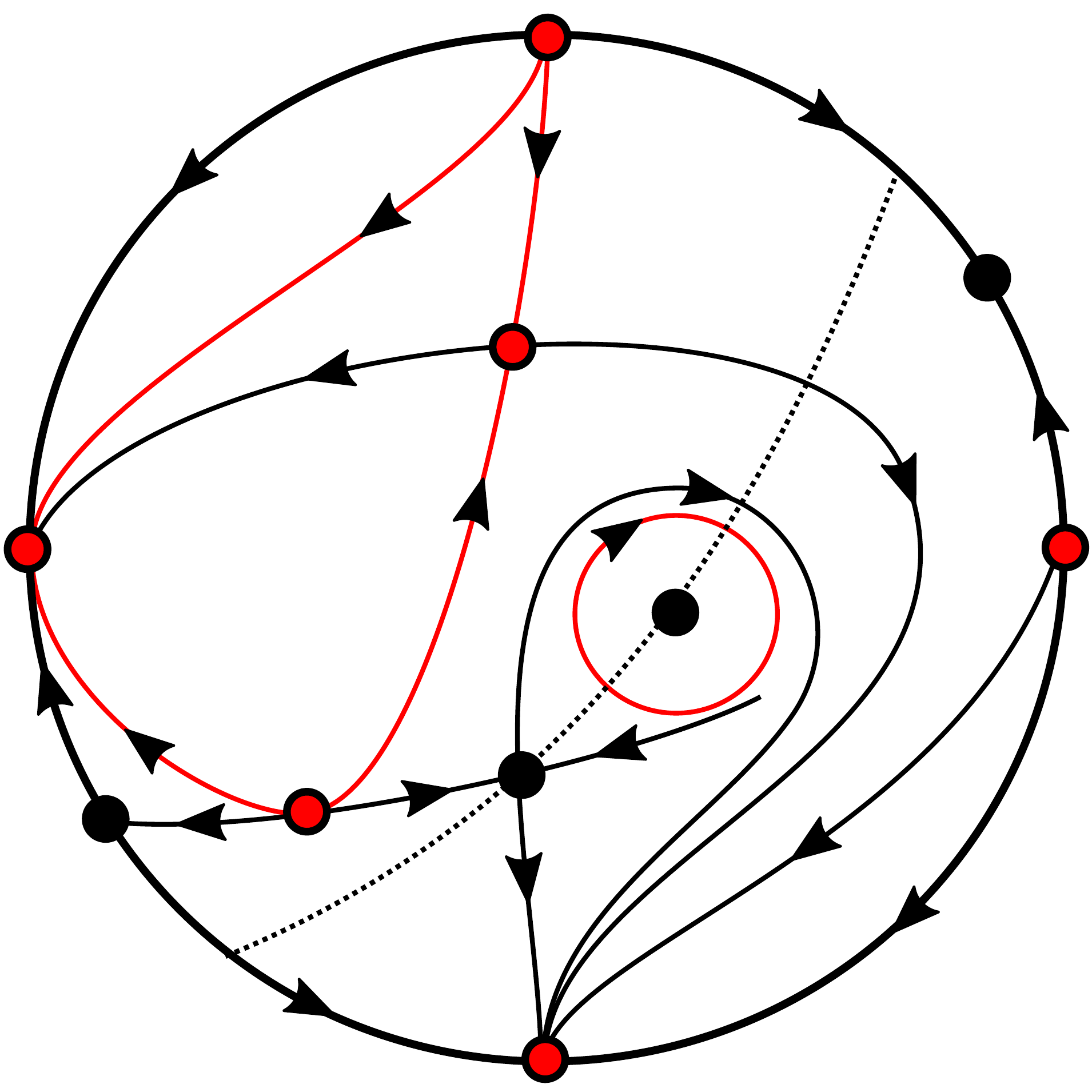}&\includegraphics[width=3.5cm]{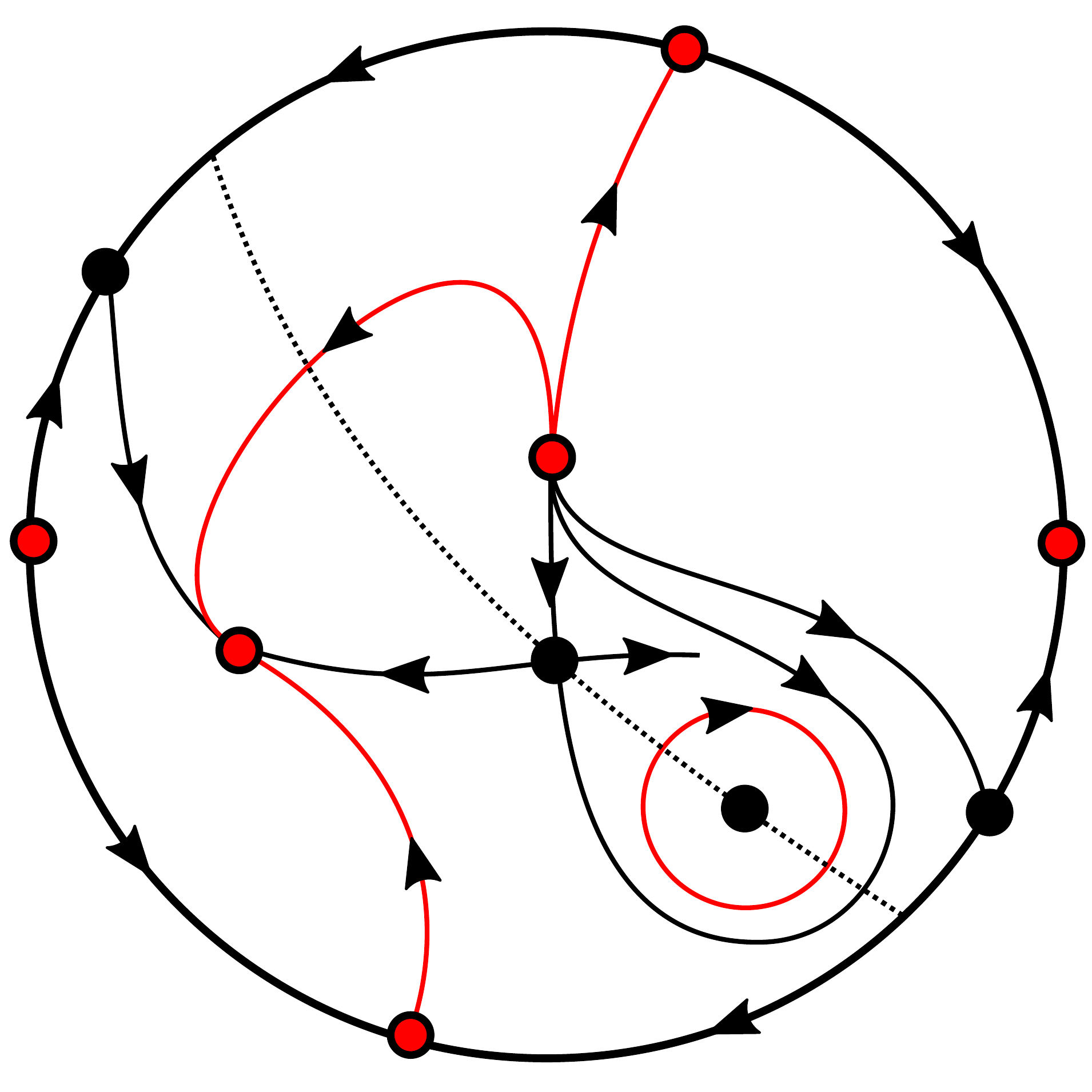}&\includegraphics[width=3.5cm]{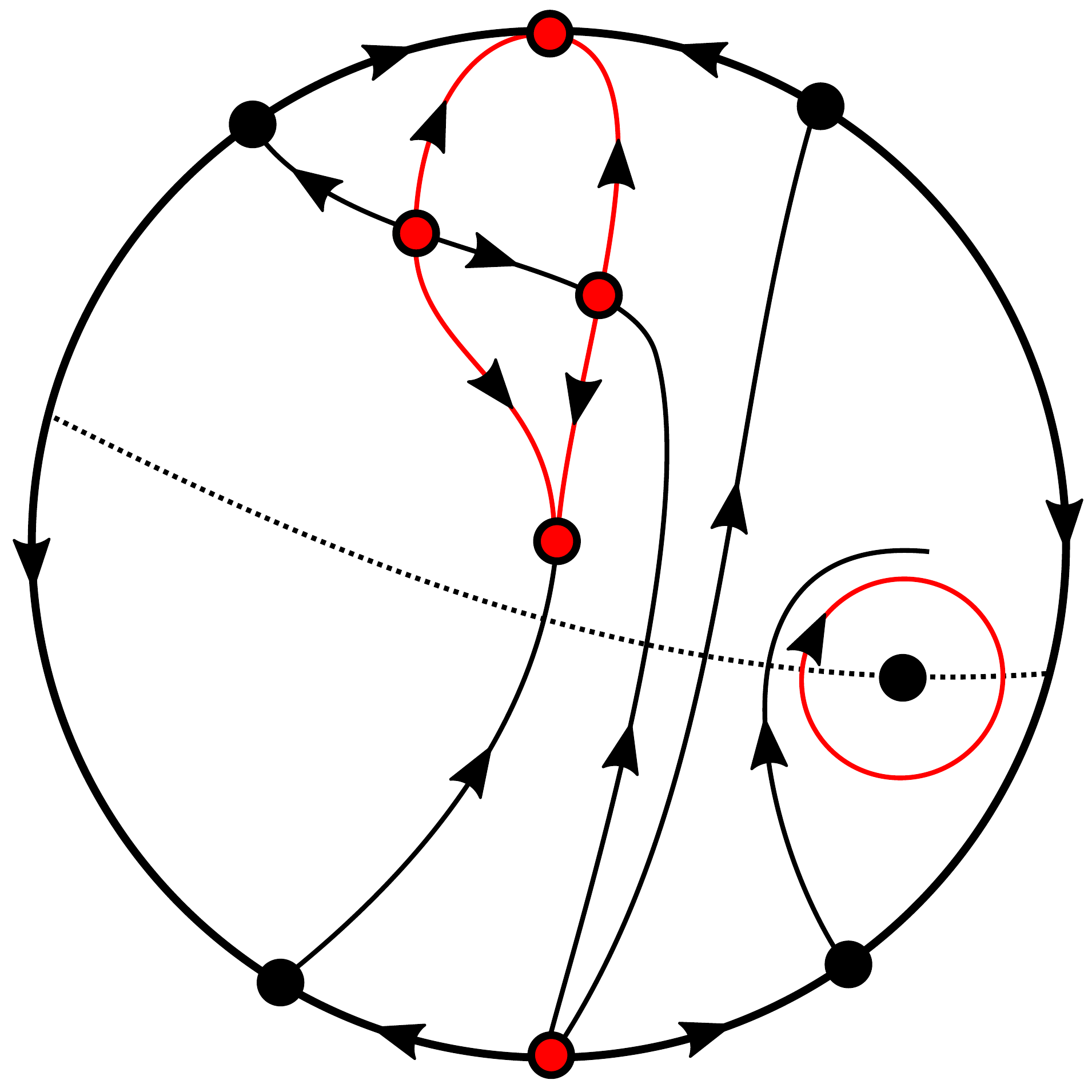}\cr
CLS&CLS5&CLS6
\end{tabular}
\caption{\footnotesize Phase portraits of the known quadratic differential systems having an algebraic limit cycle in the Poincar\'e disk. The red lines correspond to the invariant algebraic curves. The dashed lines correspond to the cofactors.}\label{F.PP}
\end{figure}

\section{Plane Cremona maps}\label{S.PVF+Cre}

\subsection{Projective vector fields}\label{S.PVF}

We shall use projective coordinates for our purposes, so we introduce in this section some basic notions on projective vector fields.

\smallskip

Let $A$, $B$ and $C$ be homogeneous polynomials of degree $m+1$ in the variables $X$,
$Y$ and $Z$. The homogeneous 1-form
\[
\omega=A\,dX+B\,dY+C\,dZ
\]
is said to be {\it projective} if $XA+YB+ZC=0$, that is, if there exist $L,M,N$ homogeneous polynomials of degree $m$ such that
\begin{equation}\label{LMN}
A=MZ-NY,\quad B=NX-LZ,\quad C=LY-MX.
\end{equation}
The triple $L, M, N$ can be thought as a homogeneous polynomial vector field in $\mathbb C^3\setminus\{0\}$ of degree $m$, which passing to $\CP^2$ provides the  projective {\it foliation} of {\it degree} $m$ defined by $\omega$, which we denote by $\mathcal F$. Equivalently, we will say that $(L, M, N)$ is a  polynomial vector field in $\CP^2$ of degree $m$.

\smallskip

The following result is well known, see \cite{D}.
\begin{lemma}\label{L.W}
If we take $\bar L=L+XW$, $\bar M=M+YW$ and $\bar N=N+ZW$, with $W$ a homogeneous polynomial of degree $m-1$, then the 1-form $\omega$ remains invariant.
\end{lemma}

Lemma \ref{L.W} tells us that $(\bar L,\bar M,\bar N)$ defines also $\omega$, i.e., $A=\bar MZ-\bar NY$, $B=\bar NX-\bar LZ$, $C=\bar LY-\bar MX$.

\smallskip

The singular points $p$ of $\omega$ (or of $\mathcal F$) are those satisfying the system of equations $A(p)=B(p)=C(p)=0$. The following result gives un upper bound for the number of singular points, see \cite{D} again.

\begin{proposition}
The number of singular points of any homogeneous polynomial vector field $(L, M, N)$ in $\CP^2$, with $L, M, N$ coprime of degree $m$, having finitely many singular points is at most $m^2+m+1$.
\end{proposition}


Let $F$ be a homogeneous polynomial of degree $n$ in $\CP^2$. We say that $F=0$ is an {\it invariant algebraic curve} of $\omega=0$ if
\begin{equation}\label{Fomega}
L\pd FX+M\pd FY+N\pd FZ=KF,
\end{equation}
where $K$ is a polynomial of degree $m-1$, called the {\it cofactor} of $F$. Euler's theorem for any homogeneous polynomial of degree $n$ gives the relation
\[
X\pd FX+Y\pd FY+Z\pd FZ=nF.
\]
From the above relation and from \eqref{Fomega} we have
\begin{equation}\label{FK}
\pd FX\left(L-\frac{KX}n\right)+\pd FY\left(M-\frac{KY}n\right)+\pd FZ\left(N-\frac{KZ}n\right)=0.
\end{equation}

\begin{remark}\label{R.cof0}
Taking $\bar L=L-KX/n$, $\bar M=M-KY/n$ and $\bar N=N-KZ/n$ we have that the cofactor of an invariant algebraic curve is zero for $\omega=\bar L(YdZ-ZdY)+\bar M(ZdX-XdZ)+\bar N(XdY-YdX)$.
\end{remark}

The affine quadratic vector field \eqref{e1} can be thought in $\CP^2$ as the projective 1-form
\[
Z^4\left[-q\left(\frac XZ,\frac YZ\right)\frac{ZdX-XdZ}{Z^2}+p\left(\frac XZ,\frac YZ\right)\frac{ZdY-YdZ}{Z^2}\right]=0
\]
of degree 3, where we have taken $(x,y)=(X/Z,Y/Z)$.
Indeed, taking $P(X, Y, Z)=Z^2p(X/Z, Y/Z)$ and $Q(X, Y, Z)=Z^2q(X/Z, Y/Z)$, this 1-form writes as
\begin{equation}\label{Eq.1form}
\omega=-ZQdX+ZPdY+(XQ-YP)dZ.
\end{equation}
The affine differential system \eqref{e1} is equivalent to the projective one defined by the  1-form \eqref{Eq.1form} of $\CP^2$. This is called the {\it projectivization} of  \eqref{e1}.
Observe that the degree of the affine system \eqref{e1} and the degree of its projectivization coincide.

\smallskip

I will be convenient for our purpose to work with the projective extension of an initial affine differential system, since we will apply to it a projective birational transformation, also known as Cremona map. At the end we will need to derive a new affine differential system from the transformed projective one. Hence we will study next the reverse operation of the projectivization.

\smallskip

From a foliation complex projective $\mathcal{F}$ of degree $m$  defined by the projective 1-form $\omega=A\,dX+B\,dY+C\,dZ$ in some suitable projective coordinate system, one easily restricts in the affine chart $Z\neq 0$  to the affine 1-form $A(x,y,1)\, dx + B(x,y,1) \, dy$ and hence to the affine differential system $\dot x= B(x,y,1),\quad \dot y= - A(x,y,1)$. The affine restriction of a projective differential system is the reverse operation of the projectivization, and reciprocally. However, notice that the degree of the affine restriction of a given projective differential system may have increased in one unit. Since the scope of this paper is dealing  with quadratic differential systems, the invariance of the degree through this operation is an important issue to handle with.

\smallskip

\begin{proposition}\label{P.InvariantLine}
A complex projective foliation  $\mathcal{F}$ of degree $m$ which is defined by a projective 1-form $\omega=A\,dX+B\,dY+C\,dZ$ restricts to an affine differential system of the same degree $m$ if and only if $\mathcal{F}$ has an invariant line $\mathcal L$, that is, there is a subpencil of the net $\{ a A + b B + c C=0 : a,b,c \in \mathbb{C} \}$ having $\mathcal L$ as a common factor.
\end{proposition}

Observe that in the case of the foliation defined by the form \eqref{Eq.1form}, which is the projectivization of \eqref{e1}, the invariant line is $Z=0$.

\begin{proof}
Suppose first that $\mathcal{F}$ has an invariant line $\mathcal L$, which can be written as $\mathcal L = \{ Z=0 \}$ by changing the projective coordinate system to a suitable one.
From \eqref{Fomega}  we infer $N=K Z$, with $K$ homogeneous of degree $m-1$. Substituting in \eqref{LMN} we obtain $A= Z A_1$, $B= Z B_1$ with $A_1$, $B_1$ homogeneous of degree $m$, that is, the subpencil $\{ a A + b B =0 : a,b \in \mathbb{C} \}$ is the one having $\mathcal L$ as common factor.
Now, when taking the affine chart $Z\neq 0$, $\omega$ restricts to the affine 1-form $A_1(x,y,1)\, dx + B_1(x,y,1) \, dy$ of degree $m$.

Conversely, suppose  $\omega=A\,dX+B\,dY+C\,dZ$, with $A$, $B$, $C$ homogeneous of degree $m+1$ restricts to $A(x,y,1)\, dx + B(x,y,1) \, dy$ with $A(x,y,1)$ and $B(x,y,1)$ of degree $m$. Writing
\begin{eqnarray*}
  A(X,Y,Z) &=& Z A_1(X,Y,Z) + A_2(X,Y) \, , \\
  B(X,Y,Z) &=& Z B_1(X,Y,Z) + B_2(X,Y) \, ,
\end{eqnarray*}
where $A_1$, $B_1$ are homogeneous in $X, Y, Z$ of degree $m$, and $A_2$, $B_2$ are homogeneous in $X, Y$ of degree $m+1$, the hypothesis implies that $A_2$, $B_2$ vanish identically. Therefore, the subpencil $\{ a A + b B =0 : a,b \in \mathbb{C} \}$ has $ Z=0 $ as common factor and we see from \eqref{Fomega} and \eqref{LMN} that  $\mathcal L= \{ Z=0 \}$ is an invariant line of $\mathcal{F}$.
\end{proof}


If we have a foliation $\mathcal F$ with an invariant line $\mathcal L$ and we want to obtain the corresponding affine differential system, we proceed as follows: first we compute $L,M,N$ from $A,B,C$ using \eqref{LMN} and use Lemma \ref{L.W} to obtain $\bar L,\bar M,\bar N$. $W$ is a homogeneous polynomial of degree $m-1$  to be fixed.


\begin{remark}\label{R.InvLine}
In the case of quadratic foliations invariant lines are easy to find.
From Remark \ref{R.cof0} an invariant line $\mathcal L:=a_1X+a_2Y+a_3Z=0$ can be seen as an invariant algebraic curve with null cofactor. So we have the equation
\begin{equation}\label{rectainv}
a_1\bar L+a_2\bar M+a_3\bar N=0.
\end{equation}
Solving this equation, which can be written as a system of equations with unknowns $a_1,a_2,a_3$ and the coefficients of $W$, we can obtain all the invariant lines of $\mathcal F$.
\end{remark}


\subsection{Local invariants}\label{S.LInv}

In a local setting, let $\omega $ be a holomorphic 1-form generating a holomorphic foliation $\mathcal{F}$ on a smooth surface $S$ on a neighborhood of a point $p \in S$, and let $\pi_p: S' :=\textrm{Bl}_{p}S \rightarrow S$ be the blow-up at $p$ with exceptional divisor $E_p= \pi_p^{-1} (p)$. The points of $E_p$ are called points in the \emph{first (infinitesimal) neighborhood} of $p$.
Taking local coordinates $x$, $y$ centered at $p$, the local ring $\mathcal{O}_{S,p}$ of germs of holomorphic functions in a neighborhood of $p$ is identified with $\C\{x,y\}$ (the ring of convergent power series in $x$ and $y$), and we write $\mathfrak{m}_p =(x,y)$ for the maximal ideal of $\mathcal{O}_{S,p}$. Suppose the foliation $\mathcal{F}$ is given by
\begin{equation}\label{Eq.hol1form}
\omega = a(x,y) dx + b(x,y) dy =0,
\end{equation}
with $a,b \in\C\{x,y\}$, that is, $\mathcal{F}$ is defined locally by the vector field $-b(x,y)\frac{\partial}{\partial x} + a(x,y)\frac{\partial}{\partial y}$.
The point $p$ is \emph{singular} if $a(p)=b(p)=0$.
Attached to each singular point $p \in \Sing(\mathcal{F})$ we consider two local invariants of $\mathcal{F}$: the \emph{algebraic multiplicity} $m(p, \mathcal{F})$ and the \emph{vanishing order} $l(p, \mathcal{F})$ of the pullback $\pi_p^{\ast} \omega $ over the exceptional divisor $E_p$.
Namely,
 \begin{itemize}
   \item $m(p, \mathcal{F}) = \min \{ \oo_p(a),  \oo_p(b) \}$, where $\oo_p (f)$ is the degree of the initial term of $f \in \mathbb{C}\{x,y\}$. In other words, $\oo_p $ is the $\mathfrak{m}_p$-adic order. Note that this definition can be extended to any $p \not\in \Sing(\mathcal{F})$ as $m(p, \mathcal{F}) = 0$;
   \item $l(p, \mathcal{F})=l$ satisfies that $\pi_p^{\ast} \mathcal{F} $ is defined locally at any $q \in  E_p$ by the 1-form $z^{-l} \pi_p^{\ast} \omega $, where $z$ is any equation for the exceptional divisor $E_p$ near $q$. Namely, either $l(p, \mathcal{F})= m(p, \mathcal{F}) +1$ when the exceptional divisor $E_p$ is not invariant by $\pi_p^{\ast} \mathcal{F} $ (in which case it is said that $p$ is \emph{dicritical}), or $l(p, \mathcal{F})= m(p, \mathcal{F})$ otherwise.
 \end{itemize}

We recall that a {\it star-node} is a singular point of a differential system whose Jacobian matrix is, up to a multiplicative constant, the identity. This is actually the simplest example of a dicritical singular point.

\smallskip

We introduce well-known analogous notions for a holomorphic germ of curve $C:f=0$ at the point $p \in S$ with $f \in \mathbb{C}\{x,y\}$: the \emph{multiplicity} of $C$ (or of $f$) at $p$ is  $m(p, C)= m(p, f) := \oo_p (f)$, and the \emph{value} $l(p, C)= l(p, f)$ of $C$ (or of $f$) at $p$ is the vanishing order of the pullback $\pi_p^{\ast} f $ over the exceptional divisor $E_p$, which equals the multiplicity $m=m(p, C)$ (see for instance \cite[3.2.1]{casas}). The \emph{strict transform} $\widetilde{C}$ of $C$ by $\pi_p$ is then defined locally at any $q \in  E_p$ by $z^{-m} \pi_p^{\ast} f$, where $z$ is any equation for the exceptional divisor $E_p$ near $q$. The interest in distinguishing between these two concepts, multiplicity and value, will become apparent when extending them to further blow-ups.

\smallskip

Let $p$ be an \emph{infinitely near} point in $S$, that is, a point lying on a surface $p \in S' $ obtained from $S$ after a sequence $\pi : S' \xrightarrow{} S$ of  blow-ups at $\sigma$ points belonging to a set $K$. A point $q \in K$ is said to \emph{precede} $p$ if $p$ belongs to the to the pullback on $S'$ of the exceptional divisor $E_q$ of the blow-up of $q$, that is, the blow-up of $q$ is needed in order to obtain $p$. The set $K$ is called \emph{cluster} of (infinitely near) points and satisfies that for any $p \in K$ $K$ contains all the points preceding $p$.
Sometimes the points lying on a surface $S$ will be called \emph{proper} points of $S$ in order to stress their difference to the infinitely near ones. If $p$ belongs to the total exceptional divisor $\pi^{-1} (O)$ of some proper point $O \in S$, we say that $p$ is \emph{infinitely near to} $O$; if $\sigma$ is the minimal number of blow-ups that are needed to obtain $p$ from $O$, then we say that $p$ lies in the $\sigma-$\emph{th (infinitesimal) neighborhood} of $O$.
%
We extend at $p$ the notions of \emph{algebraic multiplicity} $m(p, \mathcal{F}):= m(p, \pi^{\ast} \mathcal{F})$ and \emph{vanishing order} $l(p, \mathcal{F}):= l(p, \pi^{\ast} \mathcal{F})$ of a foliation $\mathcal{F}$ on $S$, and of \emph{multiplicity} $m(p, C)= m(p, f) := m(p, \widetilde{C})$  and \emph{value} $l(p, C)= l(p, f) := l(p, \pi ^{\ast} f)$ of a curve $C: f=0$ on $S$.
We say that the point $p$ is \emph{singular} (respectively, \emph{simple}) for $\mathcal{F}$ if  $m(p, \mathcal{F})> 0$ (respectively, $m(p, \mathcal{F})=1$). It holds that the vanishing order of the pullback $(\pi \circ \pi_p)^{\ast} f $ over the exceptional divisor $E_p$ equals $l(p, C)$. We say that the point $p$ lies on $C$ if $m(p, C)\neq 0$.

\begin{remark}\label{R.UniversalPropertyBlowingUp} 
Notice that these notions being invariant by local isomorphism and by using the universal property of blowing-up (\cite[3.3]{casas}), $m(p, \mathcal{F})$, $l(p, \mathcal{F})$, $m(p, C)$ and $l(p, C)$ are independent on the number $\sigma$ of blowing-ups that are performed to reach $p$ from $O$.
\end{remark}

\smallskip

The set $K$ of points (which is a union of clusters) which have been blown up to obtain $\pi : S' \xrightarrow{} S$ gives a parameterization of the set of (irreducible) exceptional components $\{ E_{p_i} \}_{p_i \in K}$ on $S'$. By a slight abuse of notation, we keep denoting by $E_p$ the strict transforms (on all intermediate surfaces) of the exceptional divisor of the blow-up at $p \in K$. We may establish a \emph{proximity relation} between the points in $K$. Namely, we say that a point $q\in K$ is \emph{proximate} to $p\in K$ if and only if $q$ belongs (as proper or infinitely near point) to the exceptional divisor $E_p$. We will denote this relation as $q\rightarrow p$. Since the total exceptional divisor of $\pi $, $E_{p_1 }+ \cdots + E_{p_{\sigma}}$, is a normal crossings divisor on $S'$ (i.e. any pair of non-disjoint components intersect transversally and there are no more than two components meeting at a point), any non-proper point $p \in K$ is proximate to at most two other points in $K$: if it is proximate to just one point, $p$ is called \emph{free}, and it is called \emph{satellite} otherwise.

\smallskip

A cluster $K$ of infinitely near points to some proper point $O \in S$ is described by means of an Enriques diagram, the proper points are represented by black-filled circles and the infinitely near ones are represented by grey-filled circles. These conventions will be used throughout this work for all the pictures depicting clusters.
An {\em Enriques diagram} is a tree, rooted on the proper point $O$, whose vertices are identified with the points in $K$, and there is an edge between $p$ and $q$ if and only if $p$ lies on the first neighborhood of $q$ or vice-versa. Moreover, the edges are drawn (as dotted arcs by convention in this work) according to the following rules:
\begin{itemize}
\item If $q$ is proximate to just one point $p$, the edge joining $p$ and $q$ is curved and, if $p \neq O$, it is tangent to the edge ending at $p$.
\item If $p$ and $q$ ($q$ in the first neighborhood of $p$) have been represented, the rest of points proximate to $p$ arising in successive blow ups are proximate to exactly two points, and they are represented on a straight half-line starting at $q$ and orthogonal to the edge ending at $q$.
\end{itemize}

The multiplicities of a curve $C$ on $S$ satisfy the so called \emph{proximity equalities} at any proper or infinitely near point $p$ (\cite[3.5.3]{casas}):
\begin{equation}\label{Eq.proxeq}
m(p,C)= \sum_{q \rightarrow p} m(q,C) ,
\end{equation}
and they are related to the values by means of the formulae (\cite[4.5.1]{casas}):
\begin{equation}\label{Eq.multvalues}
l(p,C)= m(p,C) + \sum_{p \rightarrow q} l(q,C).
\end{equation}

\begin{remark}\label{R.Seiden}
Notice that from \eqref{Eq.proxeq} it follows that the multiplicities of a curve do not increase on further blowing-ups. By a result of Seidenberg \cite{S} the same holds for foliations: if $ q \rightarrow p$ then $m(p,\mathcal{F}) \geq m(q,\mathcal{F})$.
\end{remark}

Now, returning to our setting of a projective foliation $\mathcal{F}$ on the plane $\CP^2$, $\CP^2$ will play the role of the aforementioned surface $S$: by taking suitable affine charts the algebraic projective 1-form \eqref{Eq.1form} can be written locally as an holomorphic 1-form as in \eqref{Eq.hol1form}.

\subsection{Plane Cremona maps}\label{SS.Cre}

A plane Cremona map is a birational map between two complex projective planes $\Phi : \CP^2_1 \dashrightarrow \CP^2_2$. There is a largest (Zariski-open) subset $U \subseteq \CP^2_1$ where the map $\Phi $ is defined. It satisfies that $F_{\Phi }=\CP^2_1 - U$ is a finite set of points, called \emph{fundamental points} of $\Phi $. Once projective coordinates are fixed in both planes, $\Phi $ is defined by three homogeneous polynomials $H_1$, $H_2$, $H_3$ of degree $n$ in the variables $X$, $Y$, $Z$, with no common factor. The linear system $\mathcal{H}= \{ a_1 H_1 + a_2 H_2+ a_3 H_3=0 : a_i \in \mathbb{C} \}$ defining $\Phi $ is called \emph{homaloidal net}, its members are called \emph{homaloidal curves}, and $d(\mathcal{H})=n$ is called the \emph{degree} of the map $\Phi $. It is worth noticing that this notion of degree of the homaloidal net differs from the degree of $\Phi $ as a rational map, this latter being always one, since it is generically one-to-one. Next we shall present some basic notions and results about plane Cremona maps relevant to this work and we refer the interested reader to \cite{A-C2002} for a deeper insight.

\smallskip

Any plane Cremona map $\Phi $ factorizes as the blow up of a sequence of $\sigma$ points
 $${\displaystyle \pi:S=S_{\sigma}\xrightarrow{}S_{\sigma-1}\xrightarrow{}\cdots \xrightarrow{}S_0=\CP^2_1}$$ with $S_{i+1} = \textrm{Bl}_{p_{i+1}}S_i$ for a point $p_{i+1} \in S_i$,
followed by the blow downs of a sequence of $\sigma$ (-1)-curves (curves with autointersection equal to $-1$)
$${\displaystyle \pi' :S=S'_{\sigma}\xrightarrow{}S'_{\sigma -1}\xrightarrow{}\cdots \xrightarrow{}S'_0=\CP^2_2}$$ with $S'_{i+1} = \textrm{Bl}_{q_i}S'_i$ for a point $q_{i+1} \in S'_i$, that is, the contraction of the (-1)-curve $E'_{q_{i+1}}$, which is also the exceptional divisor of the blow up at $q_{i+1}$ (see \cite[Section 2.1]{A-C2002}). Thus we have the equality of birational maps $\Phi = \pi' \circ  \pi^{-1}$, where $\pi $ and $\pi '$ are morphisms and hence they are denoted by an arrow, whereas birational maps, as $\Phi $, are denoted by a broken arrow (meaning that it is not defined everywhere).
Whenever $\sigma$ is minimal, the set $K = \{ p_1, \ldots p_{\sigma} \}$ of points which have been blown up is called the cluster \emph{base points} of $\Phi $. Then the set $K'= \{ q_1, \ldots q_{\sigma} \}$ is the cluster of base points of $\Phi ^{-1}$. Observe that the fundamental points of $\Phi $, being proper planar points, are included into the base points. Namely, the proper points in $K$ are those of $F_{\Phi }$.
Notice that some coincidences between the exceptional curves of $\pi$ and $\pi'$ may occur, namely, $E_{p_i} = E'_{q_{j}}$ for some pairs of indexes $(i,j)$. In this case, the base point $p_i$ of $\Phi $ is called \emph{non-expansive}, and accordingly $q_j$ is a non-expansive base point of $\Phi ^{-1}$. The rest of base points lacking this property are called \emph{expansive}.

\smallskip

Whenever $\sigma$ is minimal, the proper or infinitely near points of $S$ are in bijection through $\pi$ to the proper or infinitely near points of $\CP^2_1$ which are not base points of $\Phi $, and this correspondence will be denoted by $\pi _{\ast}$ or $(\pi ^{-1})_{\ast}$. Consider a proper or infinitely near point $p$ of $\CP^2_1$, and suppose $(\pi ^{-1})_{\ast}(p)$ is equal or infinitely near to the proper point $p'$ of $S$, then its image by $\Phi $ is a well defined proper point $\Phi (p) = \pi' (p')$ in $\CP^2_2$. Likewise we also state a finer notion, $\Phi _{\ast}(p)= \pi '_{\ast} (\pi ^{-1})_{\ast}(p)$, which is the point $(\pi ^{-1})_{\ast}(p)$ in $S$ regarded as a proper or infinitely near point of $\CP^2_2$ through $\pi'$. Observe that $\Phi _{\ast}(p)$ is equal or infinitely near to the point $\Phi (p)$.

\smallskip

Now, we will weight each base point $p \in K$ of $\Phi $ by a non-negative integral value $m (p, \mathcal{H})$ deeply related to the homaloidal net $\mathcal{H}$ as follows.
If $p$ is a proper point of $S:=\CP^2_1 $, without loss of generality we may assume that $p$ lies on the affine chart given by $Z\neq 0$ (otherwise perform a projective coordinate change), and write $x= \frac{X}{Z}$, $y=\frac{Y}{Z}$. Define the \emph{multiplicity} of the net $\mathcal{H}$ at $p$ as $ m = m (p, \mathcal{H})= \min\{ \oo_p(h_1), \oo_p(h_2),\oo_p(h_3) \} $ with $h_i (x,y)= H_i(x,y,1)$.
%
If $q $ is any point in the first neighborhood of $p$, consider $\pi_p: S' :=\textrm{Bl}_{p}S \rightarrow S$ the blow-up at $p$ with exceptional divisor $E_p= \pi_p^{-1} (p)$. Since the pull-back of functions induces an injective homomorphism of rings $\pi_{p}^{\ast}: \mathcal{O}_{S,p} \longrightarrow \mathcal{O}_{S',q}$, we may consider on $S'$ the net defined locally at $q$ as $\mathcal{H}_p=z^{-m} \pi_{p}^{\ast}(\mathcal{H})$ (where $z$ is any equation for the exceptional divisor $E$ near $q$), that is, generated by $ z^{-m} \pi_{p}^{\ast}(h_i)$, $i \in \{1,2,3\}$. Define $m (q, \mathcal{H}):= m (q, \mathcal{H}_p)$. Recursively this definition may be extended to any $q$ infinitely near to $p$.

\smallskip

Still associated to the homaloidal net $\mathcal{H}$, we may define at whatever (proper or infinitely near) point $p$ in $\CP^2_1 $ the \emph{value} of  $\mathcal{H}$ at $p$ as $l(p, \mathcal{H})= \min\{ l(p, h_1), l(p,h_2),l(p,h_3) \}$.
It is worth to notice that values of $\mathcal{H}$ are straightforward computed from the individual values of its three generators, whereas the computation of the multiplicities requires the common setting of the multiplicities of the generators and subsequent decision at each step of the blowing-up.

\begin{theorem}\label{T.Homaloidal}
\begin{enumerate}
\item[{\rm (1)}] The multiplicities of a homaloidal net $\mathcal{H}$ satisfy the proximity equalities at any proper or infinitely near base point $p$ :
\begin{equation}\label{Eq.proxeq_homal}
m(p,\mathcal{H})= \sum_{q \rightarrow p} m(q,\mathcal{H}) .
\end{equation}
\item[{\rm (2)}] The values of a homaloidal net $\mathcal{H}$ satisfy \begin{equation}\label{Eq.multvalues_homal}
l(p,\mathcal{H})= m(p,\mathcal{H}) + \sum_{p \rightarrow q} l(q,\mathcal{H}).
\end{equation}
Futhermore, the base points of a homaloidal net $\mathcal{H}$ are characterized as those proper or infinitely near points $p$ for which $l(p,\mathcal{H})- \sum_{p \rightarrow q} l(q,\mathcal{H})> 0$.
\item[{\rm (3)}] Homaloidal nets $\mathcal{H}$ are linear systems which have the property of being \emph{complete}: any curve $C$ of degree $n$ and having values $l(p,C) \geq l(p,\mathcal{H})$ at all base points $p$ of $\mathcal{H}$ is a homaloidal curve.
\end{enumerate}
\end{theorem}

\begin{proof}
By \cite[2.1.3]{A-C2002}, generic homaloidal curves  $C: a_1 H_1 + a_2 H_2+ a_3 H_3=0$ with $(a_1,a_2,a_3) \in V $, $V$ a (Zariski) open subset of $\mathbb{C}^{3}$), satisfy $m(p,C)= m(p,\mathcal{H})$. Hence, claim (1) follows, and, applying \eqref{Eq.multvalues}, the first assertion of claim (2) follows as well.
The rest of claim (2) comes from the recent characterization in \cite{ACAMB16} of the base points of an ideal. Notice that the base points of the homaloidal net $\mathcal{H}$, weighted by the multiplicities or the values, equals the union of weighted clusters of base points of the ideals of the stalks of the ideal sheaf generated by the homaloidal net. From this and using \cite[2.5.2]{A-C2002}, claim (3) is inferred.
\end{proof}

Given a curve $C$ in $\CP^2_1$ its \emph{image} $\Phi(C)$ is the closure of $\Phi(C - F_{\Phi })$ in $\CP^2_2$. If $\Phi(C)$ is a union of points, $C$ is called $\Phi$\emph{-contractile}. There is a maximal $\Phi$\emph{-contractile} curve, which equals $ C_{\Phi} = \bigcup \pi (E'_{q_i})$, the union running over the indexes $1 \leq i \leq \sigma$ for which $q_i$ is expansive.

\begin{lemma} \label{L.MaximalContractile}
Restricted to $\CP^2_1 - C_{\Phi}$ the map $\Phi $ is an isomorphism onto $\CP^2_2 - C_{\Phi ^{-1}}$.
\end{lemma}

\begin{proof}
Since $\Phi = \pi' \circ  \pi^{-1}$, $\Phi $ restricted to $\CP^2_1 - (F_{\Phi} \cup C_{\Phi})$ is an isomorphism (see \cite[2.1.9]{A-C2002}). We shall prove the sharper result of the statement by showing $ F_{\Phi}\subset C_{\Phi}$.
Let $p_i \in F _{\Phi}$. Since the base points infinitely near to $p_i$ constitute a cluster, $E_{p_i}$ is connected on $S$ to any $E_p$
with $p$ infinitely near to $p_i$ through a chain of exceptional divisors $E_{r_1}= E_{p_i}, \ldots, E_{r_a}= E_p $ where two consecutive elements intersect on $S$.
According to \cite[2.2.6]{A-C2002}, the maximal base points (by the ordering of being infinitely near) of a plane Cremona map are all expansive. Hence among the former points there exits an expansive $p$, for which $r_2, \ldots, r_{a-1}$ are non-expansive.
Now, applying \cite[2.6.6]{A-C2002}, the irreducible curve $\pi '(E_{p})$ on $\CP^2_2$ must go at least through one base point $q_k$ of $\Phi ^{-1}$.
Then, again, $E_{p}$ is connected on $S$ to any $E'_q$ with $q$ equal or infinitely near to $q_k$ through a chain of exceptional divisors $E_{p}, E'_{s_1}, \ldots, E'_{s_b}= E'_q$ where two consecutive elements intersect on $S$. Again by \cite[2.2.6]{A-C2002}, among the former points we may take an expansive $q$ which is minimal in the sense that all $s_1, \ldots, s_{b-1}$ are non-expansive. Notice that $E'_{s_i} = E_{t_i}$ for suitable non-expansive $t_i \in K$, $1 \leq i \leq b-1$.
Summing up, there is a chain of exceptional divisors of $\pi $, $E_{p_i}, E_{r_2}, \ldots, E_{r_{a-1}}, E_p, E_{t_1}, \ldots, E_{t_{b-1}}, E'_q$, connecting $E_{p_i}$ to $E'_q$ on $S$, where two consecutive elements intersect.
Thus the irreducible component $\pi (E'_{q})$ of $C_{\Phi}$ on $\CP^2_1$ goes through the point $p_i$, as desired.
\end{proof}

Suppose the inverse map $\Phi^{-1}$ is defined by homaloidal net $\mathcal{H}'$ spanned by the three homogeneous polynomials $H'_1$, $H'_2$, $H'_3$ in the variables $U$, $V$, $W$, with no common factor. If $C: G(X,Y,Z)=0$ is not contractile, then its \emph{direct image} $\Phi_{\ast}(C)$ is the curve in $\CP^2_2$ defined from the equation $G(H'_1, H'_2, H'_3)=0$ after deleting all the $\Phi ^{-1}$-contractile curves. 
If $\mathcal{F}$  is a projective foliation defined by the projective homogeneous 1-form $\omega (X,Y,Z) =A\,dX+B\,dY+C\,dZ$, then its \emph{direct image} $\Phi_{\ast}\mathcal{F}$ is the foliation in $\CP^2_2$ defined from the 1-form $(\Phi ^{-1})^{\ast}(\omega)$:  $ \omega (H'_1, H'_2, H'_3)= \overline{A}\,dU+\overline{B}\,dV+\overline{C}\,dW$ after deleting all common factors from $\overline{A}$, $\overline{B}$ and $\overline{C}$. These common factors happen to be all $\Phi ^{-1}$-contractile curves, since $\Phi $ is isomorphism outside $C_{\Phi}$ (Lemma \ref{L.MaximalContractile}).

\begin{remark}\label{R.DirectImagFol}
Notice that for whatever factorization of $\Phi = \pi' \circ  \pi^{-1}$, where $\pi $ and $\pi '$ are compositions of blowing-ups we have the equality of foliations $\pi ^{\ast} \mathcal{F} = (\pi ')^{\ast} \Phi_{\ast}\mathcal{F}$ on $S$.
\end{remark}

The action of a general plane Cremona map on curves is completely known (see \cite[Lemma 2.9.3]{A-C2002}). We highlight the most relevant results for our purpose in the following

\begin{lemma}\label{L.CremonaOnCurves}
If $C$ is a curve of  degree $d(C)$, then the degree of its direct image $\Phi_{\ast}(C)$ is
\[
d(C) d(\mathcal{H})- \sum_{i=1}^{\sigma} m(p_i,C) m(p_i,\mathcal{H}).
\]
If moreover $C$ has no contractile components, the multiplicities of $\Phi_{\ast}(C)$ can also be predicted:
\[
m(q_k,\Phi_{\ast}(C))=d(C) m(q_k,\mathcal{H}')- \sum_{i=1}^{\sigma}  m(p_i,C) e_{q_k}(p_i,\mathcal{H}),
\]
where the $e_{q_k}(p_i,\mathcal{H})$ are natural numbers algorithmically determined from the vector encoding the numerical features of $\Phi$, $\left(d(\mathcal{H}); m(p_1,\mathcal{H}),\ldots ,m(p_{\sigma},\mathcal{H})  \right)$.
\end{lemma}


Forthcoming Section \ref{SS.TransfFolByCre} is devoted to the study of the action of plane Cremona maps on foliations.

\subsection{Quadratic plane Cremona maps}

In this work we shall focus on plane Cremona maps of degree $2$, called \emph{quadratic}. Notice that any plane Cremona map may be expressed as the composition of quadratic transformations (see \cite[Theorem 8.4.3]{A-C2002}).

\smallskip

In virtue of \cite[Section 2.8]{A-C2002} any quadratic plane Cremona map with homaloidal net $\mathcal{H}$ has three base points, say them $p_1$, $p_2$, $p_3$, and they all are simple, that is, they have multiplicity $\mu (p_i, \mathcal{H})=1$ for all $i \in \{1,2,3\}$. As a consequence of Theorem \ref{T.Homaloidal}, according to the proximity relations between these base points, quadratic maps can be classified into three types:
\begin{enumerate}
\item[{\rm (C1)}] An \emph{ordinary} quadratic plane Cremona map: all three base points are proper planar points and hence there is no proximity relation between them.
\begin{figure}[h]
\centering
\includegraphics[width=12cm]{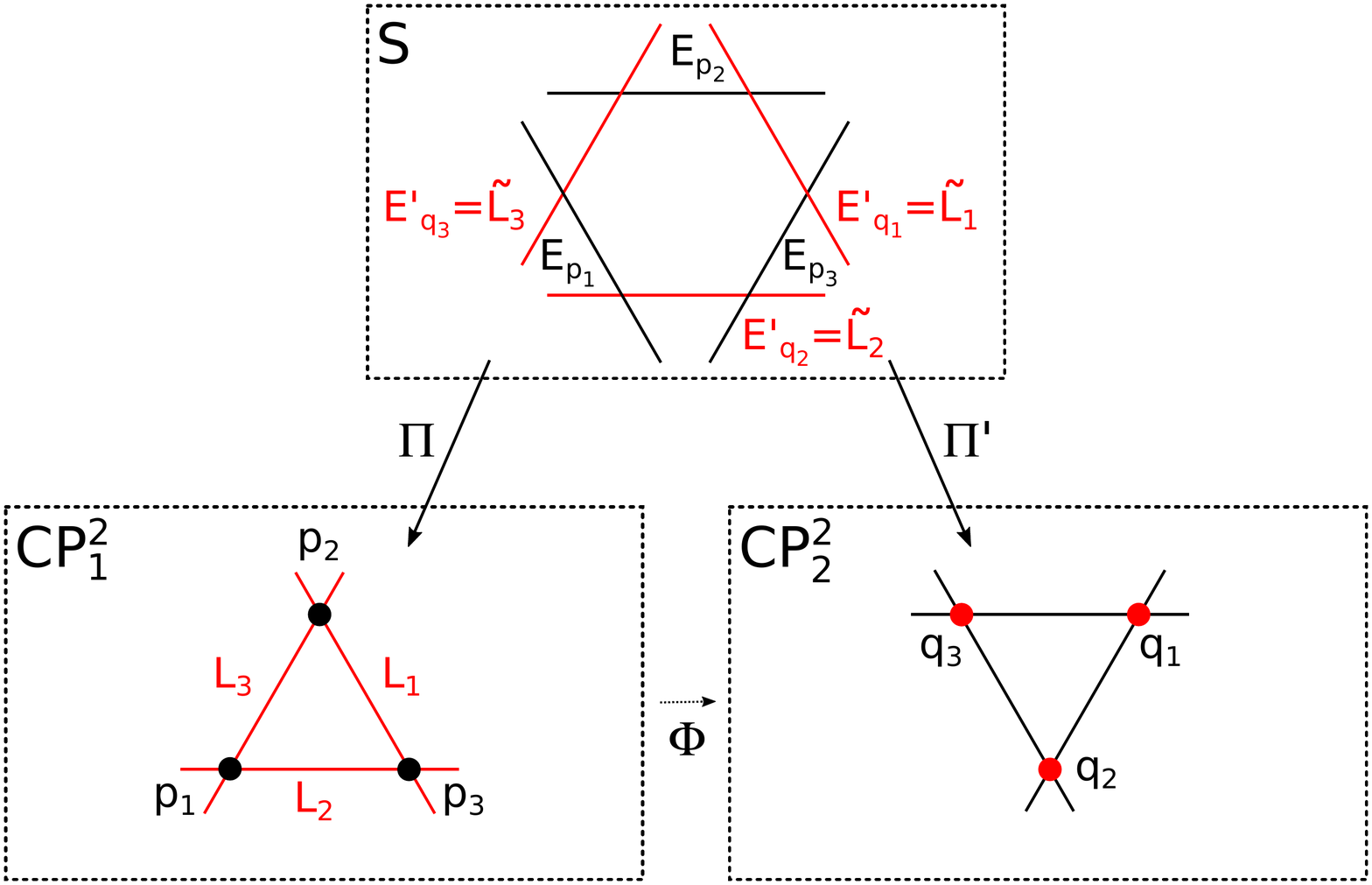}
\caption{\footnotesize The ordinary plane Cremona map.}\label{F.C1}
\end{figure}
Any ordinary quadratic Cremona map factorizes as the blow-up at $K= \{p_1, p_2, p_3\}$, followed by the blow-downs of the strict transforms $\widetilde{L}_{i} = E'_{q_i}$ of the three lines $L_i := p_j  p_k$, with $\{i,j,k\} = \{1,2,3\}$. Then $\{q_1, q_2, q_3\}$ are the base points of the inverse. See Figure \ref{F.C1}.

\item[{\rm (C2)}] A quadratic plane Cremona map with exactly two proper planar base points, $p_1$ and $p_2$. The third base point $p_3$ lies on the first neighborhood of one of them, suppose of $p_2$ (i.e. $p_3$ lies on the exceptional component $E_{p_2}$ of blowing up $p_2$) and there is only one proximity relation, namely $p_3 \rightarrow p_2$.
\begin{figure}[h]
\centering
\includegraphics[width=12cm]{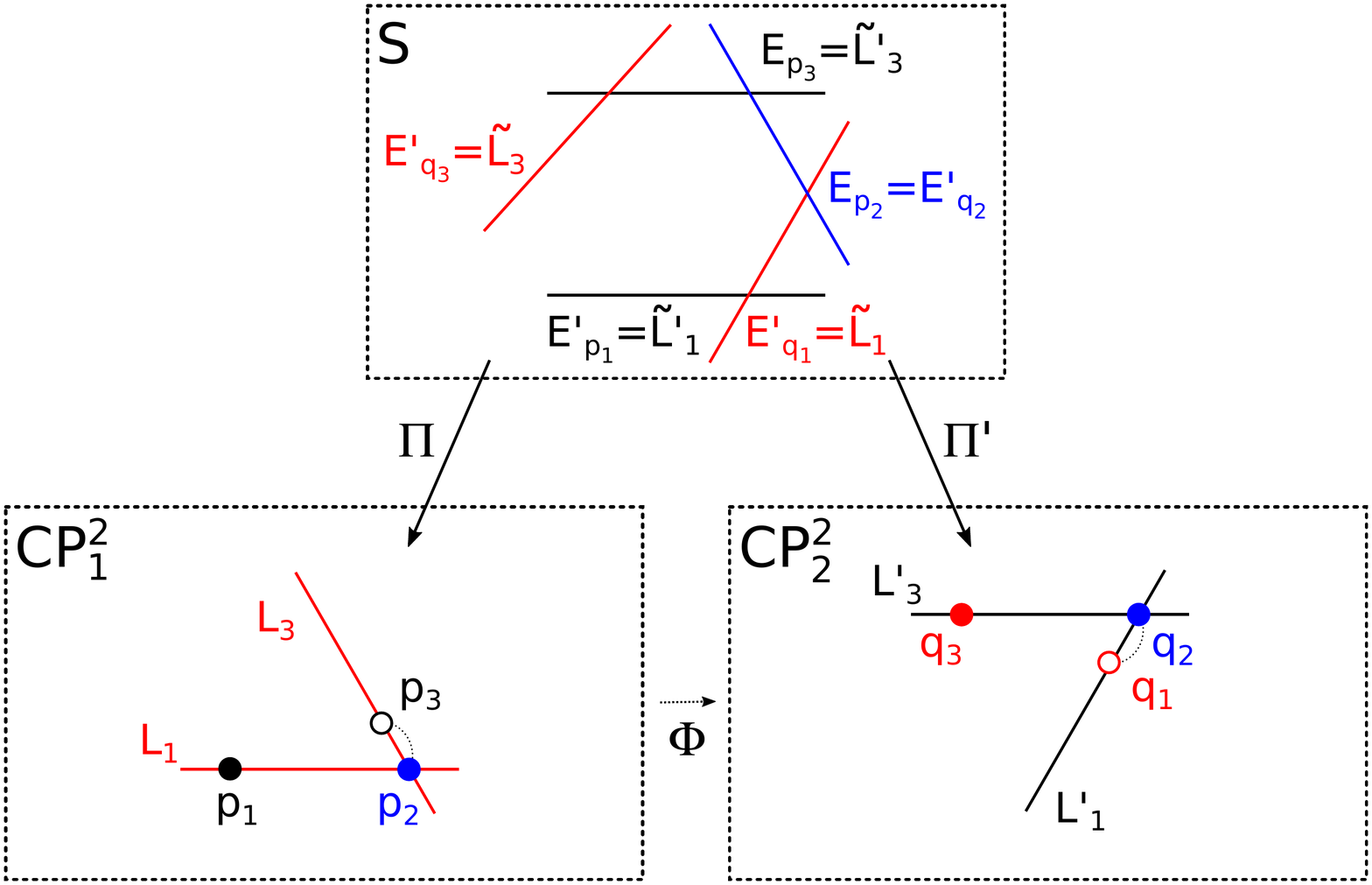}
\caption{\footnotesize A quadratic plane Cremona map with exactly two proper planar base points (I).}\label{F.C2}
\end{figure}
\begin{figure}[h]
\centering
\includegraphics[width=12cm]{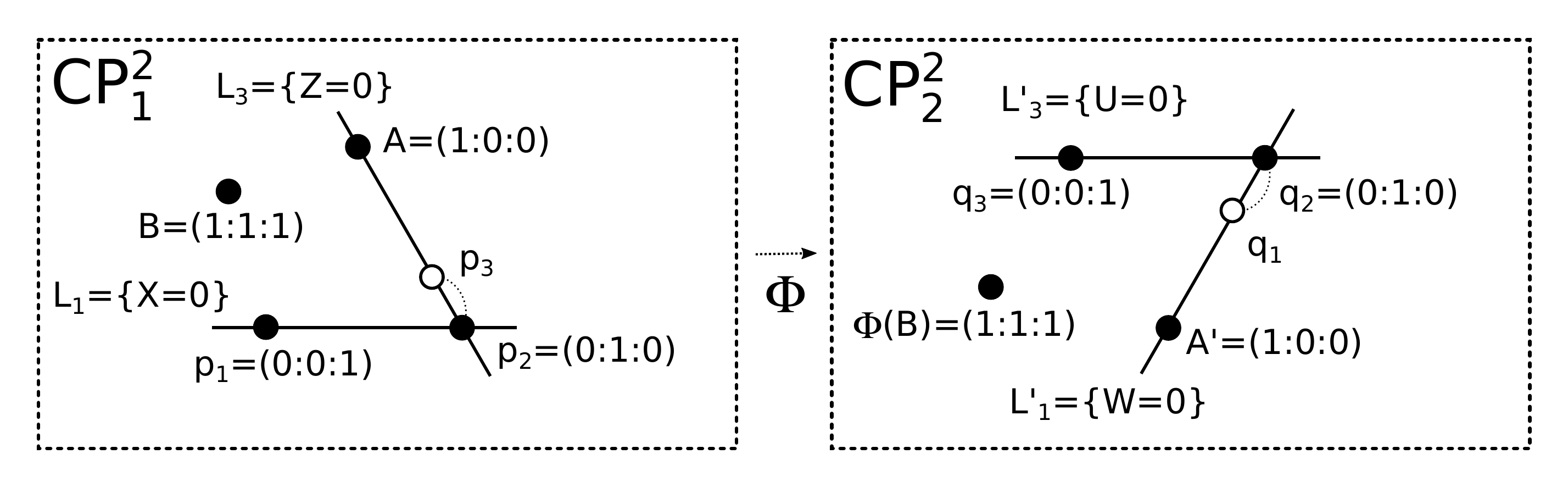}
\caption{\footnotesize A quadratic plane Cremona map with exactly two proper planar base points (II).}\label{F.C2b}
\end{figure}
Any quadratic Cremona map of type (C2) factorizes as the blow-up at $K= \{p_1, p_2, p_3\}$, followed by the blow-downs of the strict transforms $\widetilde{L}_{i} = E'_{q_i}$ of the lines $L_i := p_2  p_i$, for $i \in \{1,3 \}$ (where $p_2  p_3$ is the unique line going through $p_2$ such that its multiplicity at $p_3$ is one) and of the exceptional divisor $E_{p_2} = E'_{q_2}$. Then $\{q_1, q_2, q_3\}$ are the base points of the inverse, they satisfy $q_1 \rightarrow q_2$ and $ E_{p_i}$ are the strict transforms $\widetilde{L'_{i}}$ of the lines $L'_i := q_2  q_i$, for $i \in \{1,3 \}$ (where $q_2  q_1$ is the unique line going through $q_2$ such that its multiplicity at $q_1$ is one). See Figures \ref{F.C2} and \ref{F.C2b}.

\item[{\rm (C3)}] A quadratic plane Cremona map with a unique proper planar base point, $p_1$. A second base point, $p_2$, lies on the first neighborhood of $p_1$, and the third base point $p_3$ lies on the first neighborhood of $p_1$ and it is only proximate to $p_2$. That is, $p_2 \rightarrow p_1$ and $p_3 \rightarrow p_2$ are all the proximity relations.
\begin{figure}[h]
\centering
\includegraphics[width=12cm]{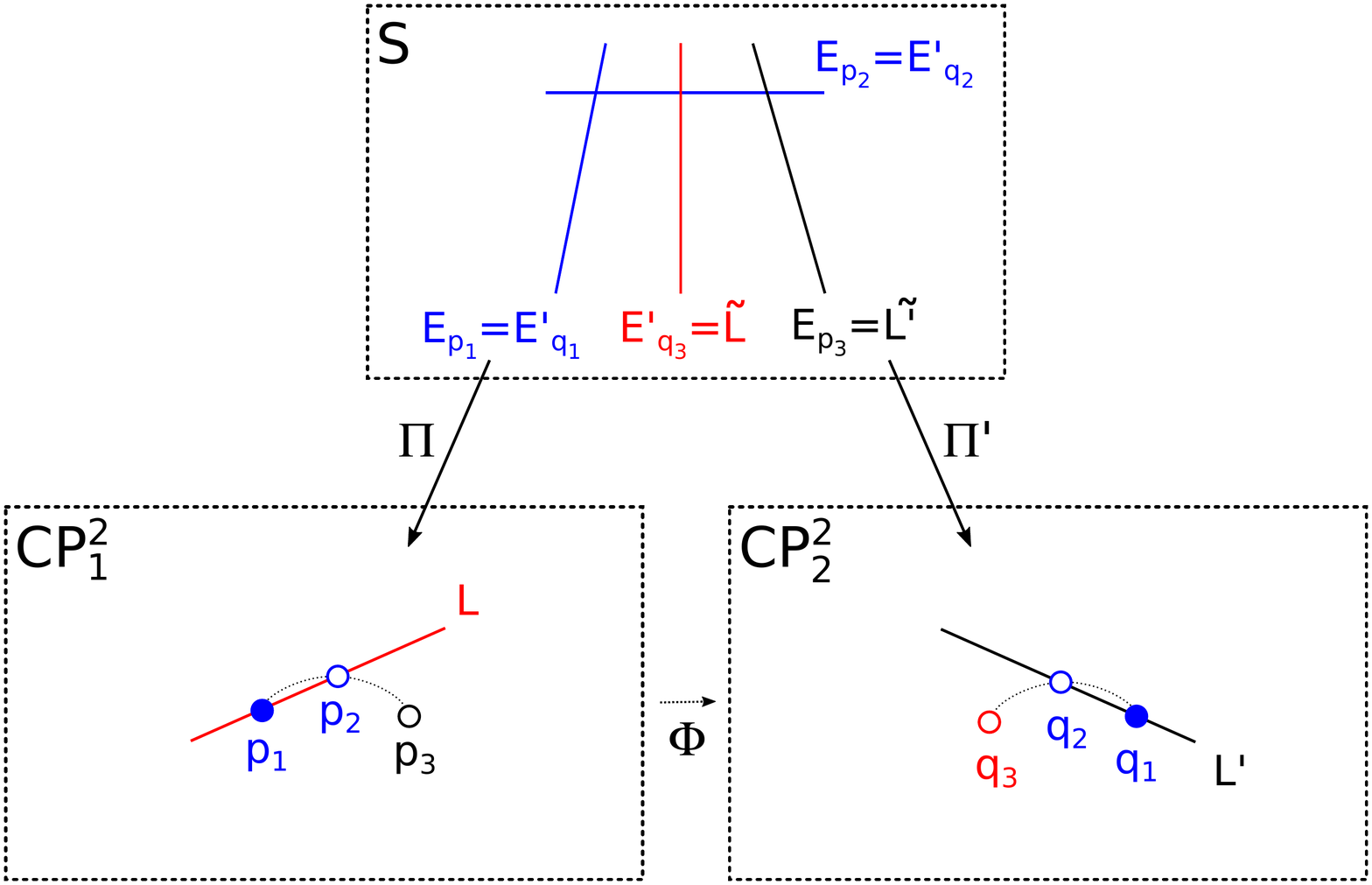}
\caption{\footnotesize A quadratic plane Cremona map with a unique proper planar base point (I).}\label{F.C3}
\end{figure}
\begin{figure}[h]
\centering
\includegraphics[width=12cm]{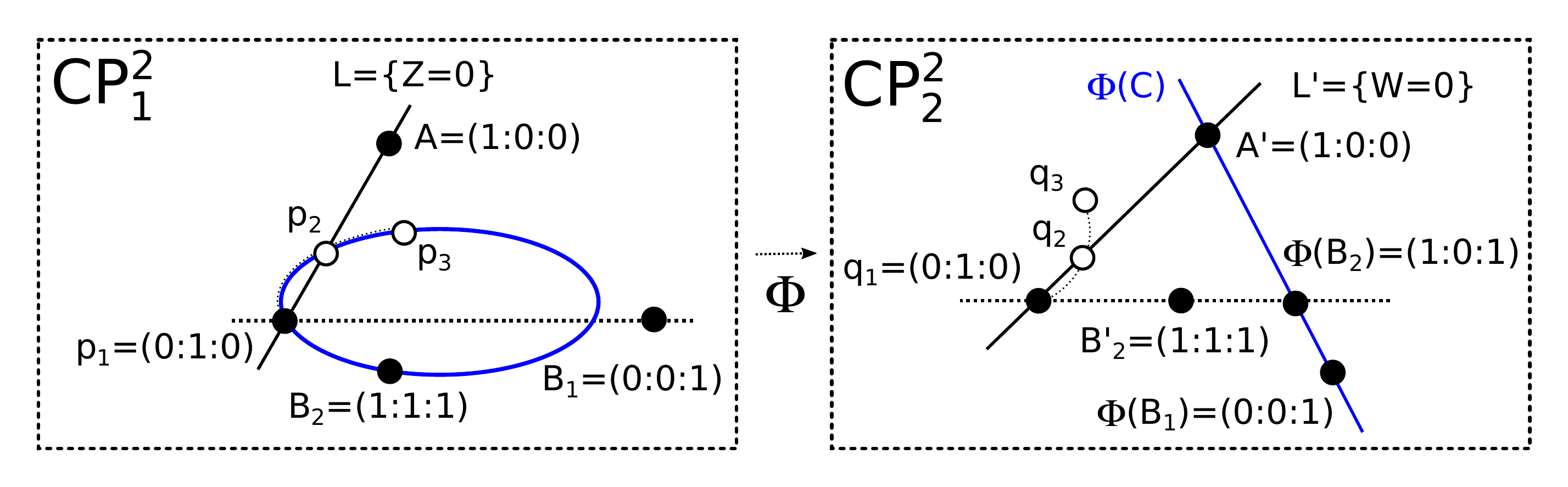}
\caption{\footnotesize A quadratic plane Cremona map with a unique proper planar base point (II).}\label{F.C3b}
\end{figure}
Any quadratic Cremona map of type (C3) factorizes as the blow-up at $K= \{p_1, p_2, p_3\}$, followed by the blow-downs of the strict transform $\widetilde{L} = E'_{q_3}$ of the line $L := p_1  p_2$ (where $p_1  p_1$ is the  unique line going through $p_1$ such that its multiplicity at $p_2$ is one) and of the exceptional divisors $E_{p_2} = E'_{q_2}$ and $E_{p_1} = E'_{q_1}$. Then $\{q_1, q_2, q_3\}$ are the base points of the inverse, they satisfy $q_2 \rightarrow q_1$ and $q_3 \rightarrow q_2$, and $ E_{p_3}$ is the strict transform $\widetilde{L'}$ of the line $L':= q_1  q_2$ (where $q_1  q_2$ is the unique line going through $q_1$ such that its multiplicity at $q_2$ is one). See Figures \ref{F.C3} and \ref{F.C3b}.
\end{enumerate}

Notice that no projectivity of the plane can transform a quadratic plane Cremona map of one type into another of a different type: indeed, a projectivity sends proper points to proper points, and each of the three types of quadratic maps has a different number of proper base points. Next result classifies quadratic plane Cremona maps under \emph{projective equivalence}, that is, two maps are equivalent if you obtain one from the other by composing with suitable planar projectivities, in both the departure $\CP^2_1$ and target $\CP^2_2$ planes. We will prove that any quadratic Cremona map is projectively equivalent to a unique distinguished representative belonging to one of the three types listed above.

\begin{proposition}\label{P.Cre}
The quadratic plane Cremona maps can be classified under projective equivalence into type \textrm{(C1)}, \textrm{(C2)} or \textrm{(C3)}, and each type has a distinguished normal form:
\begin{enumerate}
\item[{\rm (C1)}] An ordinary quadratic Cremona transformation can be written, up to projective equivalence, as
\begin{equation}\label{C1}
(U,V,W)= (YZ,XZ,XY).
\end{equation}
Its inverse is
\begin{equation}\label{C1i}
(X,Y,Z)=(VW,UW,UV).
\end{equation}
Notice that both transformations coincide, so the representative of the class {\rm (C1)} is an involution.
\item[{\rm (C2)}] A quadratic Cremona transformation with only two proper base points can be written, up to projective equivalence, as
\begin{equation}\label{C2}
(U,V,W)= (XZ,YZ,X^2).
\end{equation}
Its inverse is
\begin{equation}\label{C2i}
(X,Y,Z)=(UW,VW,U^2).
\end{equation}
Again the representative of the class {\rm (C2)} is an involution.
\item[{\rm (C3)}] A quadratic Cremona transformation with a unique proper base point can be written, up to projective equivalence, as
\begin{equation}\label{C3}
(U,V,W)= (XZ,YZ+cX^2,Z^2).
\end{equation}
for some convenient $c\neq0$. Its inverse is
\begin{equation}\label{C3i}
(X,Y,Z)=(UW,VW-cU^2,W^2).
\end{equation}
Notice that any $c\neq0$ gives a representative of the class {\rm (C3)}.
\end{enumerate}
\end{proposition}

\begin{proof}
First of all, notice that any quadratic plane Cremona $\Phi$ map must fall into one of the three types (C1), (C2) or (C3) which have been described above. Indeed, these types comprise all the possibilities of proximity relations between the three simple base points of $\Phi$, according to the proximity equalities (1) in Theorem  \ref{T.Homaloidal}.

\smallskip

By taking a projective coordinate system in $\CP^2_1$ with $p_1 =(0:0:1)$, $p_2=(1:0:0)$, $p_3= (0:1:0)$ and whatever unit point $A:=(1:0:0)$, and the projective coordinate system in $\CP^2_2$ satisfying $q_1= (0:0:1)$, $q_2=(1:0:0)$, $q_3= (0:1:0)$ with unit point $\Phi (A)=(1:1:1)$, and by applying \cite[2.8.2]{A-C2002}, it follows that any quadratic Cremona map of type (C1) is projectively equivalent to the ordinary quadratic map in \eqref{C1}.

\smallskip

Start now with a quadratic Cremona map $\Phi$ of type (C2) with homaloidal net $\mathcal{H}$. Fix in $\CP^2_1$ the projective coordinate system $p_1= (0:0:1)$, $ p_2=(0:1:0)$, $A:=(1:0:0)$ being any point on the line $L_3:=p_2  p_3$ different from $p_2$, and take whatever suitable unit point $B:=(1:1:1)$. In $\CP^2_2$, fix the coordinate system $q_2= (0:1:0)$, $ q_3=(0:0:1)$, and, for the moment, $A':=(1:0:0)$
is any proper point on the line $L'_1:= q_2  q_1$, and take $B':=(1:1:1)=\Phi (B)$ as the unit point.
We will refine throughout the proof the election of the third coordinate point in $\CP^2_2$.

\smallskip

All the conics $C: a_1 XZ + a_2 YZ+ a_3 X^2= 0$ satisfy
$l(p,C) \geq l(p,\mathcal{H})$ at any base point $p$ of $\Phi$. Hence they all are homaloidal curves, according to Theorem \ref{T.Homaloidal}(3), and hence the homaloidal net $\mathcal{H}$ of $\Phi$ is $\mathcal{H}= \{ a_1 XZ + a_2 YZ+ a_3 X^2= 0 : a_i \in \mathbb{C} \}$. Hence $\Phi$ is defined by equations
\[
(U,V,W)= \left(a_1 XZ+ a_2 YZ+ a_3 X^2, b_1 XZ+ b_2 YZ+ b_3 X^2, c_1 XZ+ c_2 YZ+ c_3 X^2\right)
\]
for some $a_i, \, b_i, \, c_i \in \mathbb{C}$. Now, from the description of the (C2) type, we must impose the conditions $\Phi (p_2 p_1)= q_2$ and $\Phi (p_2 p_3)= q_3$: $\Phi (\{ X =0\})= (a_2:b_2:c_2)= (0:1:0)$ and $\Phi (\{ Z =0\})= (a_3:b_3:c_3)= (0:0:1)$, that is, $a_2=c_2=0$ and $a_3=b_3=0$. Notice that those equations for $\Phi$ also give $\Phi (\{ c_3X + c_1 Z =0\})\subseteq \{ w =0\}$, which is impossible unless $c_1=0$.

\smallskip

At this point, we have already proved that $\Phi$ has equations
\[
(U,V,W)= \left(a_1 XZ, b_1 XZ+ b_2 YZ, c_3 X^2\right)
\]
for some $a_1, \, b_1, \,b_2, \, c_3 \in \mathbb{C}$. We can still refine the choice of $A'$ in order to infer the desired normal form of \eqref{C2}.
Observe that the image by $\Phi$ of the line $y=0$ is the line $a_1 V-b_1 U = 0$. Hence choosing $A'$ in the intersection between the lines $q_2  q_1$ and  $\Phi (\{ Y=0\})$ is equivalent to saying that $\Phi (\{Y=0\})= \{ a_1 V -b_1 U = 0 \} = \{ V=0 \}$, that is, $b_1=0$.
Finally we impose the last condition $\Phi (B)=B'$, giving $\Phi (1:1:1)=(a_1:b_2:c_3)=(1:1:1)$, from which we infer the desired normal form of \eqref{C2}.

\smallskip

Consider a quadratic Cremona map $\Phi$ of type (C3) with homaloidal net $\mathcal{H}$. Fix in $\CP^2_1$ the projective coordinate system $p_1=(0:1:0)$, $A:=(1:0:0)$ any proper point on the line $L = p_1 p_2$, and whatever suitable third coordinate point $B_1:= (0:0:1)$ and unit point $B_2:=(1:1:1)$. In $\CP^2_2$, fix the coordinate system $q_1= (0:1:0)$, $B'_{1}=(0:0:1)=\Phi (B_1)$ and, for the moment, $A':=(1:0:0)$ is any proper point on the line $L':= q_1  q_2$, and take as unit point $B'_{2}:=(1:1:1)$ any suitable point on the line $q_1 \Phi (B_2)$. We will refine throughout the proof the election of the unit point in $\CP^2_1$ and of the third coordinate and unit points in $\CP^2_2$.

\smallskip

Choose a suitable $c \in \mathbb{C}$ with $c\neq 0$ so that the conic $C: YZ + c X^2=0$ goes through the infinitely near point $p_3$, that is, $m(p_3,C) =1$.
Then, all the conics $C: a_1 XZ + a_2 (YZ + c X^2)+ a_3 Z^2= 0$  satisfy
$l(p,C) \geq l(p,\mathcal{H})$ at any base point $p$ of $\Phi$.
Hence they all are homaloidal curves, according to Theorem \ref{T.Homaloidal}(3), and hence the homaloidal net $\mathcal{H}$ of $\Phi$ is $\mathcal{H}= \{ a_1 XZ + a_2 (YZ + c X^2)+ a_3 Z^2= 0 : a_i \in \mathbb{C} \}$. Hence $\Phi$ is defined by equations
\[
\begin{split}
U= a_1 XZ+ a_2 (YZ + c X^2)+ a_3 Z^2 \\
V= b_1 XZ+ b_2 (YZ + c X^2)+ b_3 Z^2 \\
W= c_1 XZ+ c_2 (YZ + c X^2)+ c_3 Z^2
\end{split}
\]
for some $a_i, \, b_i, \, c_i \in \mathbb{C}$. Now, from the description of the (C3) type, we must impose the condition $\Phi (p_1 p_2)= q_1$: $\Phi (\{ Z =0\})= (a_2:b_2:c_2)= (0:1:0)$, that is, $a_2=c_2=0$. Notice that those equations for $\Phi$ give $\Phi (\{ a_1X + a_3 Z =0\})\subseteq \{ U=0\}$, which is impossible unless $a_3=0$, and $\Phi (\{ c_1X + c_3 Z =0\})\subseteq \{ W =0\}$, which is impossible unless $c_1=0$.
Now, owing to the way we have chosen $B'$, we may impose $\Phi (B)=B'$ and this gives $\Phi (0:0:1)=(0:b_3:c_3)=(0:0:1)$, that is, $b_3=0$.

\smallskip

At this point, we have already proved that $\Phi$ has equations
\[
(U,V,W)= \left(a_1 XZ, b_1 XZ+ b_2 (YZ + c X^2), c_3 Z^2\right)
\]
for some $a_1, \, b_1, \,b_2, \, c_3 \in \mathbb{C}$, and a suitable non-zero $c$. We can still refine some choices in order to infer the desired normal form of \eqref{C3}.
Observe that the image by $\Phi$ of the conic
$C: YZ + c X^2=0$ is the line $a_1 v -b_1 u = 0$. Hence choosing $A'$ in the intersection of the lines $L'$ and $\Phi (C)$ is equivalent to saying that $\Phi (C)= \{ a_1 V -b_1 U = 0 \} = \{ V=0 \}$, that is, $b_1=0$.
Now, we can also refine the choice of the unit point $B'_{2}$ in $\CP^2_2$ and we take as $B'_{2}:=(1:1:1)$ the unique point on the line $q_1 \Phi (B_2)$ which makes that the coordinates of $\Phi (B_2)$ are exactly $(1:1+c:1)$. Then, imposing $\Phi (1:1:1)=(a_1:b_2 (1+c):c_3)=(1:1+c:1)$ gives the desired normal form of \eqref{C3}, in which $c$ is determined by the relative position of $p_3$ in the third infinitesimal neighborhood of $p_1$.

\smallskip

Finally we will show that we can still refine the choice of the unit point $B_{2}$ in $\CP^2_1$ to achieve that any two quadratic plane Cremona maps in normal form \eqref{C3} with different parameter $c$ are projectively equivalent. Indeed, taking $B_2:=(1:1:1)$ on the conic $C: YZ + c X^2=0$, we get $c=-1$.
\end{proof}

\begin{figure}[h]
\centering
\begin{tabular}{ccc}
\includegraphics[width=3cm]{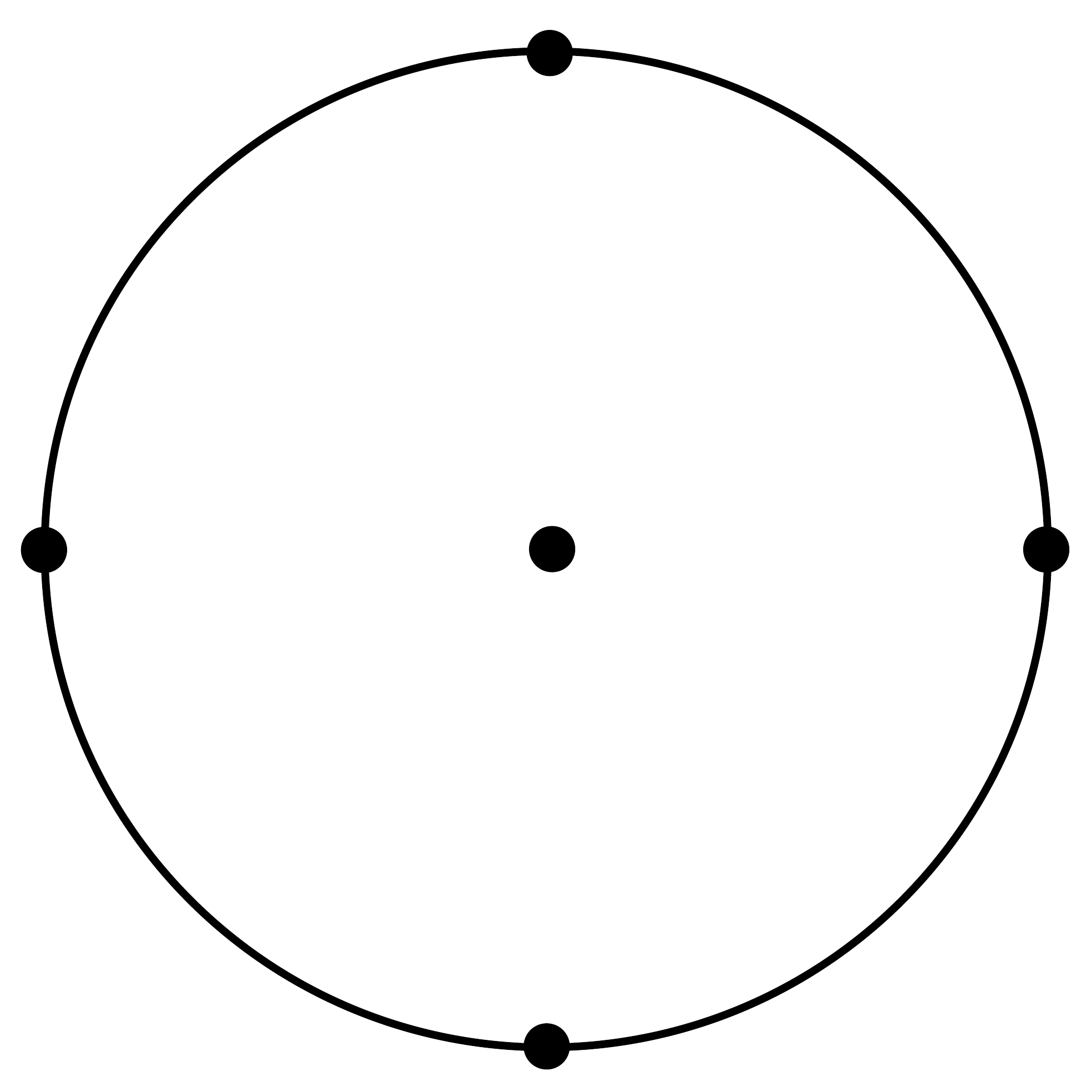}&\includegraphics[width=3cm]{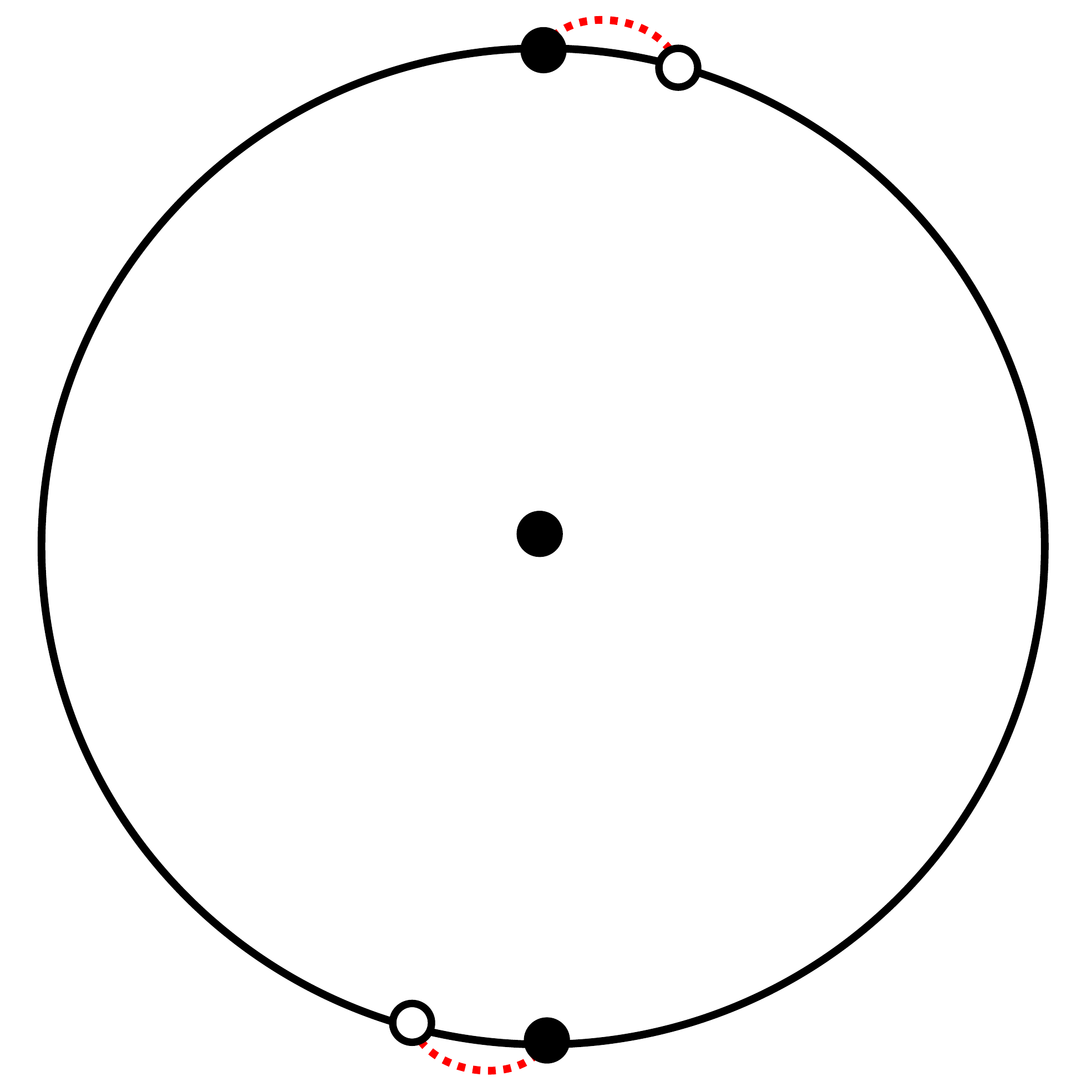}&\includegraphics[width=3cm]{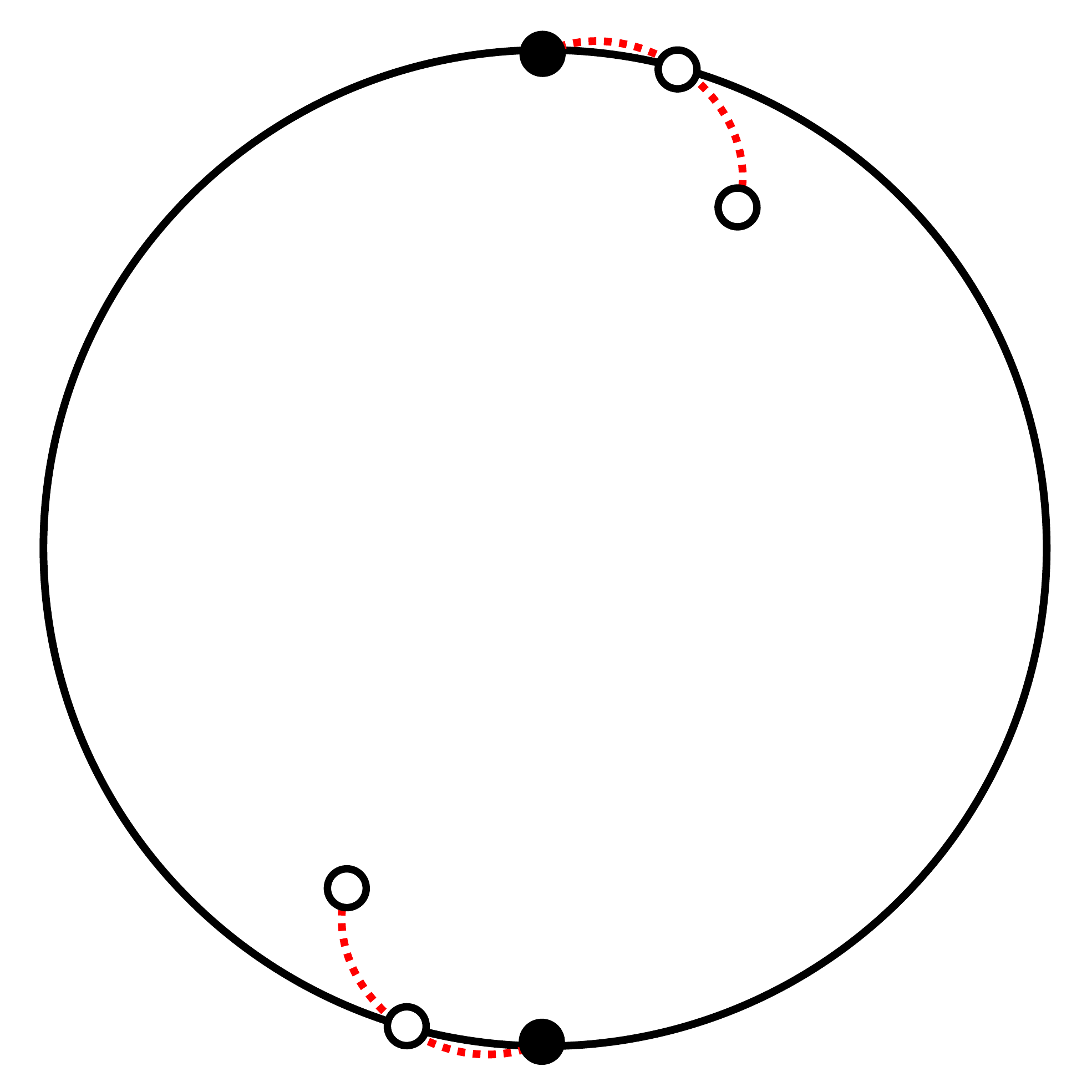}\cr
(a)&(b)&(c)
\end{tabular}
\caption{\footnotesize The three base points of the different types of Cremona maps (C1), (C2) and (C3) listed in Proposition \ref{P.Cre}, from left to right, drawn in the Poincar\'{e} disk. Black dots are proper points, while white dots are infinitely near singular points.}
\end{figure}

Still more interestingly the proof of Proposition \ref{P.Cre} gives the following existential result on quadratic plane Cremona maps.
\begin{proposition}\label{P.ExistCre}
Let $\{ p_1, p_2, p_3 \}$ be a cluster of (infinitely near) points in the plane. There exists a quadratic plane Cremona map with base points $p_1, p_2, p_3$ if and only if $p_1$, $p_2$ and $p_3$ are not aligned.
\end{proposition}
\begin{proof}
If $p_1$, $p_2$ and $p_3$ are not aligned, then the construction of the proof of \ref{P.Cre} can be carried out and the the desired quadratic plane Cremona map.
Otherwise, if $p_1$, $p_2$ and $p_3$ lie on a line $L$ and such a quadratic map would exist, then $L$ would cut its homaloidal net of conics in 3 points, resulting in contradiction with Bezout's Theorem.
\end{proof}

\section{Transforming differential systems by plane Cremona maps}\label{SS.TransfFolByCre}

In this section we will describe the effect of applying quadratic plane transformations on foliations and on curves. The effect of an ordinary quadratic Cremona map is already known (see \cite{MP}). Since the base points of the ordinary map are all proper points in the plane and they are all expansive, the case of a quadratic map of type (C1) is easier to handle. In fact, the motivation for introducing in subsection \ref{S.LInv} the generalized (to infinitely near points) notions of algebraic multiplicity and vanishing order on exceptional divisors was to provide the suitable tools for describing the effect of a general quadratic Cremona map acting on any foliation.


Consider a plane Cremona map between two complex projective planes $\Phi : \CP^2_1 \dashrightarrow \CP^2_2$, and suppose $\mathcal{F}$ and $C$ a are a projective foliation and a curve in $\CP^2_1$, respectively.
The following lemma is a version of \cite[Lemma 1, pg. 278]{MP}, where the action of the plane Cremona map on curves is formulated in more precise terms (see our previous Lemma \ref{L.CremonaOnCurves}).

\begin{lemma}\label{L.MP}
Let $\Phi $ be an ordinary quadratic plane Cremona map, and suppose $p_1$, $p_2$, and $p_3$ are its proper base points, and $q_1$, $q_2$, and $q_3$ are the proper base points of its inverse $\Phi ^{-1}$, named according to Figure \ref{F.C1}.
If $C$ is a curve of  degree $d(C)$, then the degree of its direct image $\Phi_{\ast}(C)$ is
\[
2d(C)- \sum_{i=1}^3 m(p_i,C).
\]
If moreover $C$ has no contractile components, the multiplicities of $\Phi_{\ast}(C)$ can also be predicted:
\[
m(q_k,\Phi_{\ast}(C))=d(C)-m(p_i,C)-m(p_j,C),\quad i\neq j\neq k\in\{1,2,3\}.
\]
If $\mathcal{F}$  is a foliation of degree $d(\mathcal{F})$, then the degree of the foliation
$\Phi _{\ast}\mathcal{F}$ (with isolated singularities) is equal to
\[
2(d(\mathcal{F})+1)- \sum_{i=1}^3\ell(p_i,\mathcal{F}).
\]
Furthermore
\[
\ell(q_k,\Phi _{\ast}\mathcal{F})=d(\mathcal{F})+2-\ell(p_i,\mathcal{F})-\ell(p_j,\mathcal{F}),\quad i\neq j\neq k\in\{1,2,3\}.
\]
\end{lemma}


We shall use this result to prove its generalization to whatever quadratic plane Cremona map, no matter its type. As a previous step we will need a technical result which describes the local behavior of the action of plane Cremona maps on foliations at any non-base (proper or infinitely near) point, including at points on its maximal contractile curve $C_{\Phi}$.

\begin{lemma}\label{L.LocalCrem}
Let $\mathcal{F}$ be a projective planar foliation and suppose $p$ is a (proper or infinitely near) point in the plane, not being a base point of the plane Cremona map $\Phi $. 
Then it holds
\[
\ell( p,\mathcal{F})= \ell(\Phi _{\ast}( p),\Phi _{\ast} \mathcal{F}).
\]
\end{lemma}

\begin{proof}

Since $p$ is not a base point of $\Phi $, $(\pi^{-1}) _{\ast}( p)$ is a proper or infinitely near point on $S$, and by  Remark \ref{R.UniversalPropertyBlowingUp} we have $\ell( p,\mathcal{F})= \ell((\pi^{-1}) _{\ast}( p),\pi^{\ast}\mathcal{F})$. Similarly and using the equalities $\pi ^{\ast} \mathcal{F} = (\pi ')^{\ast} \Phi_{\ast}\mathcal{F}$ (Remark \ref{R.DirectImagFol}) and $\Phi _{\ast}(p)= \pi '_{\ast} (\pi ^{-1})_{\ast}(p)$,
it holds
\[
\ell(\Phi _{\ast}( p),\Phi _{\ast}\mathcal{F})= \ell((\pi '^{-1}) _{\ast}( \Phi _{\ast}( p)),(\pi ')^{\ast}\Phi _{\ast}\mathcal{F}))= \ell((\pi^{-1}) _{\ast}( p),\pi^{\ast}\mathcal{F}) =\ell( p,\mathcal{F}).
\]
\end{proof}

Notice that the hypothesis of the previous Lemma \ref{L.LocalCrem} includes non-base point lying on the maximal contractile curve $C_{\Phi}$ (cf. Lemma \ref{L.MaximalContractile}).

\begin{theorem}\label{P.MP}
Let $\Phi $ be any  quadratic plane Cremona map, and suppose $p_1$, $p_2$, and $p_3$ are its base points, and $q_1$, $q_2$, and $q_3$ are the base points of the inverse $\Phi^{-1}$, named according a suitable ordering.
If $C$ is a curve of  degree $d(C)$, then the degree of its direct image $\Phi _{\ast}(C)$ is
\[
2d(C)- \sum_{i=1}^3 m(p_i,C).
\]
If moreover $C$ has no contractile components, then
\[
m(q_k,\Phi _{\ast}(C))=d(C)-m(p_i,C)-m(p_j,C),\quad i\neq j\neq k\in\{1,2,3\}.
\]
If $\mathcal{F}$  is a foliation of degree $d(\mathcal{F})$, then the degree of the foliation
$\Phi _{\ast}\mathcal{F}$ (with isolated singularities) is equal to
\[
2(d(\mathcal{F})+1)- \sum_{i=1}^3 \ell(p_i,\mathcal{F}).
\]
Furthermore
\[
\ell(q_k,\Phi _{\ast}\mathcal{F})=d(\mathcal{F})+2-\ell(p_i,\mathcal{F})-\ell(p_j,\mathcal{F}),\quad i\neq j\neq k\in\{1,2,3\}.
\]
\end{theorem}

\begin{proof}
The assertion on curves comes from applying together \cite[2.9.3]{A-C2002}, \cite[2.8.7]{A-C2002} and \cite[2.8.8]{A-C2002}.
We shall prove the claim on foliations by factorizing the quadratic plane Cremona map $\Phi $ as the composition of ordinary ones, and then applying previous Lemma \ref{L.MP} to each of them.

\smallskip

If $\Phi $ is of type (C1) we are done by Lemma \ref{L.MP}. So, assume first $\Phi $ is of type (C2).
We name the base points such that $p_1$ and $p_3$ are the proper base points of $\Phi $ and  $p_2$ is infinitely near to $p_1$, and that $q_1$ and $q_3$ are the proper base points of $\Phi^{-1}$ and $q_2$ is infinitely near to $q_1$.
According to \cite[8.5.1]{A-C2002}, the map $\Phi $ factorizes as the composition of two ordinary ones,
$\Phi = \Phi _{2}\circ \Phi _{1}$, where $\Phi _{1}$ is a map of type (C1) whose ordinary base points are $p_1$, $p_3$ and whatever proper point $q$ not on $C_{\Phi}$, and $\Phi _{2}^{-1}$ is a map of type (C1) whose ordinary base points are $q_1$, $q_3$ and  $\Phi (q)=\Phi _{\ast}(q)$.
Denote by $a_1$, $a_2$ and $a_3$ the proper base points of $\Phi ^{-1}_{1}$, with $\Phi ^{-1}_{1} (a_2  a_3)= p_1$, $\Phi ^{-1}_{1} (a_1  a_3)= q$ and $\Phi ^{-1}_{1} (a_1  a_2)= p_3$. The infinitely near point $p_2$ is mapped to $a_4:={\Phi _{1}}_{\ast}(p_2)$, which is a proper point lying on the line $a_2 a_3$, and it is different from $a_2$ and $a_3$.
This point $a_4$ must be a base point of $\Phi _{2}$, since the line $q_1 q_3$ is contracted to $p_2$ by $\Phi ^{-1}$ and hence to $a_4$ by $\Phi ^{-1}_{2} = \Phi _{1}\circ \Phi ^{-1}$.
On the other side, since $\Phi _{2} (a_1  a_3) = (\Phi  \circ \Phi ^{-1}_{1}) (a_1  a_3)= \Phi (q)$, it follows invoking \cite[4.2.5]{A-C2002} that the common base points of $\Phi _{2}$ and $\Phi ^{-1}_{1} $ are $a_1$ and $a_3$. Thus, the base points of $\Phi _{2}$ are $a_1$, $a_3$ and $a_4$, with $\Phi _{2} (a_1  a_3)= \Phi (q)$, $\Phi _{2} (a_1  a_4)= q_3$, $\Phi _{2} (a_3  a_4)= q_1$. Moreover $a_2 a_3=a_3 a_4$ and $ a_2 = {\Phi ^{-1}_{2}}_{\ast} (q_2)$.
\smallskip

Applying Lemma \ref{L.MP} to the Cremona map $\Phi _{1}$ we obtain
\begin{equation*}
\begin{split}
d ({\Phi _1}_{\ast}(\mathcal{F}))&= 2(d(\mathcal{F})+1)- \ell(p_1,\mathcal{F})- \ell(p_3,\mathcal{F})- \ell(q,\mathcal{F}) ,\\
\ell(a_1,{\Phi _1} _{\ast}\mathcal{F})&= d(\mathcal{F})+2-\ell(q,\mathcal{F})-\ell(p_3,\mathcal{F}),\\
\ell(a_2,{\Phi _1} _{\ast}\mathcal{F})&= d(\mathcal{F})+2-\ell(p_1,\mathcal{F})-\ell(p_3,\mathcal{F}), \\
\ell(a_3,{\Phi _1} _{\ast}\mathcal{F})&= d(\mathcal{F})+2-\ell(p_1,\mathcal{F})-\ell(q,\mathcal{F}).
\end{split}
\end{equation*}
From this and applying Lemma \ref{L.MP} again, now to the Cremona map $\Phi _{2}$, we infer that ${\Phi }_{\ast}\mathcal{F}= {\Phi _2}_{\ast} ( {\Phi _1}_{\ast}\mathcal{F})$ has degree
\begin{equation*}
\begin{split}
2 \left( 2(d(\mathcal{F})+1)- \ell(p_1,\mathcal{F})- \ell(p_3,\mathcal{F})- \ell(q,\mathcal{F})  \right) + 2
- \ell(a_1,{\Phi _1} _{\ast}\mathcal{F}) -\ell(a_3,{\Phi _1} _{\ast}\mathcal{F}) & \\ - \ell(a_4,{\Phi _1} _{\ast}\mathcal{F})=
 2(d(\mathcal{F})+1)- \ell(p_1,\mathcal{F}) - \ell(p_2,\mathcal{F})- \ell(p_3,\mathcal{F}), \\
\end{split}
\end{equation*}
since $\ell(a_4,{\Phi _1} _{\ast}\mathcal{F})= \ell({\Phi _{1}}_{\ast}(p_2),{\Phi _1} _{\ast}\mathcal{F})=\ell(p_2,\mathcal{F})$ by applying Lemma\ref{L.LocalCrem}.
Furthermore, from Lemma \ref{L.LocalCrem} again we have $\ell(q_2,\Phi _{\ast}\mathcal{F})=  \ell(a_2,{\Phi _1} _{\ast}\mathcal{F})$ and hence
\begin{equation*}
\begin{split}
\ell(q_1,\Phi _{\ast}\mathcal{F})& =d({\Phi _1} _{\ast}\mathcal{F})+2-\ell(a_3,{\Phi _1} _{\ast}\mathcal{F})-\ell(a_4,{\Phi _1} _{\ast}\mathcal{F})= d(\mathcal{F})+2-\ell(p_2,\mathcal{F})-\ell(p_3,\mathcal{F}), \\
\ell(q_2,\Phi _{\ast}\mathcal{F})& =  \ell(a_2,{\Phi _1} _{\ast}\mathcal{F})=d(\mathcal{F})+2-\ell(p_1,\mathcal{F})-\ell(p_3,\mathcal{F}), \\
\ell(q_3,\Phi _{\ast}\mathcal{F})& =d({\Phi _1} _{\ast}\mathcal{F})+2-\ell(a_1,{\Phi _1} _{\ast}\mathcal{F})-\ell(a_4,{\Phi _1} _{\ast}\mathcal{F})=
d(\mathcal{F})+2-\ell(p_1,\mathcal{F})-\ell(p_2,\mathcal{F}).
\end{split}
\end{equation*}


\smallskip

Assume now $\Phi $ is of type (C3).
We name the base points such that $p_1$ is the proper base point of $\Phi$, $p_2$ is infinitely near to $p_1$ and $p_3$ is infinitely near to $p_2$, and $q_1$ is the proper base point of $\Phi^{-1}$, $q_2$ is infinitely near to $q_1$ and $q_3$ is infinitely near to $q_2$.
According to \cite[8.5.2]{A-C2002}, the map $\Phi $ factorizes as
$\Phi = \Phi _{2}\circ \Phi _{1}$, where $\Phi _{1}$ is a map of type (C2) whose base points are $p_1$, $p_3$ and whatever proper point $q$ not aligned with $p_1$ and $p_3$ (that is, $q$ is not on $C_{\Phi}$), and $\Phi _{2}^{-1}$ is a map of type (C2) whose base points are $q_1$, $q_3$ and  $\Phi (q)=\Phi _{\ast}(q)$.
Denote by $a_1$, $a_2$ and $a_3$ the base points of $\Phi ^{-1}_{1}$; $a_1$, $a_3$ are proper points and $a_2$ is infinitely near to $a_1$.
The infinitely near point $p_3$ is mapped to $a_4:={\Phi _{1}}_{\ast}(p_3)$, which is a proper point lying on the line $a_1 a_3$, and it is different from $a_1$ and $a_3$.
This point $a_4$ must be a base point of $\Phi _{2}$, since the line $q_1 q_2$ is contracted to $p_3$ by $\Phi ^{-1}$ and hence to $a_4$ by $\Phi ^{-1}_{2} = \Phi _{1}\circ \Phi ^{-1}$.
On the other side, since $a_1 a_3$ and $a_1 a_2$ are contractile lines by $\Phi ^{-1}_{1}$ mapping to $p_2$ and $q$ respectively, they are also contractile lines by $\Phi _{2} = \Phi \circ  \Phi ^{-1}_{1}$ mapping to $q_2$ and $\Phi (q)$ respectively. Invoking \cite[4.2.5]{A-C2002} it follows that the common base points of $\Phi _{2}$ and $\Phi ^{-1}_{1} $ are $a_1$ and $a_2$. Thus, the base points of $\Phi _{2}$ are $a_1$, $a_2$ and $a_4={\Phi _{1}}_{\ast}(p_3)$. Moreover $ a_3 = {\Phi ^{-1}_{2}}_{\ast} (q_3)$.

\smallskip

Applying twice the result we have just proved for quadratic Cremona maps of type (C2), and using from Lemma \ref{L.LocalCrem} that $\ell(a_4,{\Phi _1} _{\ast} \mathcal{F})= \ell(p_3,\mathcal{F})$ and $\ell(q_3,\Phi _{\ast} \mathcal{F})=  \ell(a_3,{\Phi _1} _{\ast}\mathcal{F})$,
we infer that
\begin{equation*}
\begin{split}
d({\Phi }_{\ast}\mathcal{F}) &= 2 \left( 2(d(\mathcal{F})+1)- \ell(p_1,\mathcal{F})- \ell(p_2,\mathcal{F})- \ell(q,\mathcal{F})  \right) + 2
- \ell(a_1,{\Phi _1} _{\ast}\mathcal{F}) -\ell(a_2,{\Phi _1} _{\ast}\mathcal{F}) \\ & - \ell(a_4,{\Phi _1} _{\ast}\mathcal{F})=
 2(d(\mathcal{F})+1)- \ell(p_1,\mathcal{F}) - \ell(p_2,\mathcal{F})- \ell(p_3,\mathcal{F}), \\
\ell(q_1,\Phi _{\ast}\mathcal{F})& =d({\Phi _1} _{\ast}\mathcal{F})+2-\ell(a_2,{\Phi _1} _{\ast}\mathcal{F})-\ell(a_4,{\Phi _1} _{\ast}\mathcal{F})= d(\mathcal{F})+2-\ell(p_2,\mathcal{F})-\ell(p_3,\mathcal{F}), \\
\ell(q_2,\Phi _{\ast}\mathcal{F})& =d({\Phi _1} _{\ast}\mathcal{F})+2-\ell(a_1,{\Phi _1} _{\ast}\mathcal{F})-\ell(a_4,{\Phi _1} _{\ast}\mathcal{F})=
d(\mathcal{F})+2-\ell(p_1,\mathcal{F})-\ell(p_3,\mathcal{F}), \\
\ell(q_3,\Phi _{\ast}\mathcal{F})&=  \ell(a_3,{\Phi _1} _{\ast}\mathcal{F})=d(\mathcal{F})+2-\ell(p_1,\mathcal{F})-\ell(p_2,\mathcal{F}).
\end{split}
\end{equation*}

\end{proof}

\subsection{Transforming quadratic foliations by quadratic plane Cremona maps}

In this work we focus on foliations $\mathcal{F}$  of degree $d(\mathcal{F})=2$ or, equivalently, on quadratic polynomial differential systems. From Theorem \ref{P.MP}, we infer a geometric characterization of the invariance of the degree of quadratic foliations by the action of the quadratic plane Cremona maps.

This problem was already tackled in \cite{CD2015}: a projective foliation $\mathcal{F}$ is called \emph{numerically invariant} under the action a plane Cremona map $\Phi$ if the degree of $\Phi _{\ast}\mathcal{F}$ equals the degree of the original $\mathcal{F}$. In \cite{CD2015} they prove that any quadratic foliation numerically invariant under the action of an ordinary quadratic Cremona map is transversely projective, and they give normal forms in case of numerically invariant pairs $(\Phi , \mathcal{F})$ where the map $\Phi$ is quadratic and the foliation $\mathcal{F}$ is projective quadratic.

In this setting we provide sharper results: we geometrically characterize numerically invariant quadratic pairs, and in forthcoming sections we will be interested in normal forms of numerically invariant pairs $(\Phi , \mathcal{F})$ where the map $\Phi$ is quadratic and the foliation $\mathcal{F}$ is affine quadratic. Recall from Proposition \ref{P.InvariantLine} that a general projective quadratic foliation does not restrict to any affine quadratic foliation.
It is worth to notice that, although the degree of a foliation is a \emph{global} feature, the characterization of this paper will be given in terms of \emph{local} features of the initial foliation: a direct inspection at the multiplicities and eigenvalues of the singular points suffices to elucidate the existence of some quadratic plane Cremona map capable to transform a quadratic foliation maintaining the degree invariant.


\begin{corollary}\label{T.Main.G}
Let  $\mathcal{F}$  be a complex projective foliation of degree $d(\mathcal{F})=2$. Then any  quadratic plane Cremona map $\Phi $ transforms the foliation into a foliation  ${\Phi }_{\ast}\mathcal{F}$ of degree lower than or equal to $6$.
Moreover, the transformed projective foliation ${\Phi }_{\ast}\mathcal{F}$  is quadratic if and only if
the three base points $p_1$, $p_2$ and $p_3$ of $\Phi $ satisfy one of the following conditions, in which we assume that $\{i,j,k\}=\{1,2,3\}$ and that the point is not dicritical unless it is explicitly mentioned:
\begin{enumerate}[{\rm(1)}]
\item $m(p_i, \mathcal{F})=3$, $m(p_j, \mathcal{F})=m(p_k, \mathcal{F})=0$ and $p_i$ is dicritical;
\item $m(p_i, \mathcal{F})=3$, $m(p_j, \mathcal{F})=1$ and $m(p_k, \mathcal{F})=0$;
\item $m(p_i, \mathcal{F})=2$, $m(p_j, \mathcal{F})=1$, $m(p_k, \mathcal{F})=0$ and $p_i$ is dicritical;
\item $m(p_i, \mathcal{F})=1$, $m(p_j, \mathcal{F})=2$, $m(p_k, \mathcal{F})=0$ and $p_i$ is dicritical;
\item $m(p_i, \mathcal{F})=m(p_j, \mathcal{F})=1$, $m(p_k, \mathcal{F})=0$ and $p_i$ and  $p_j$ are dicritical;
\item $m(p_i, \mathcal{F})=2$ and  $m(p_j, \mathcal{F})=m(p_k, \mathcal{F})=1$;
\item $m(p_i, \mathcal{F})=m(p_j, \mathcal{F})=m(p_k, \mathcal{F})=1$ and $p_i$ is dicritical.
\end{enumerate}
\end{corollary}


\begin{proof}
The different cases follow by direct application of Theorem \ref{P.MP}, knowing that both multiplicities and vanishing order are non-negative integers. We distinguish the different cases according to the possibilities that the vanishing order coincides with the multiplicity at some proper or infinitely near point (see Section \ref{S.LInv}).
\end{proof}

The result of Corollary \ref{T.Main.G} shows that the invariance of the degree is due to local properties of the foliation on the base points of the Cremona map.
We remark that the base points do not need to be singular points of the system. Indeed they can be either regular points or infinitely near singular points.
Next example shows that complex non-real singular points cannot be neglected when searching for Cremona transformations which provide new differential systems of a specific degree.

\begin{example}
Consider the Yablonski differential system \eqref{YabQS}. Recall that it has two complex non-real finite singular points. An affine complex transformation moves one of these two points to the origin. At infinity, we have the singular point $(0:1:0)$ and two complex non-real infinite singular points. The quotient of the eigenvalues at this infinite singular point is $2$. We apply the Cremona transformation (C2), see Proposition \ref{P.Cre}, to obtain a cubic projective 1-form that can be brought to the quadratic complex differential system
\[
\begin{split}
\dot x=&3 c (a+b)x+4 y+\frac{(c (4 c^2+1) (a+b)-(2 c^2-1) \Delta ) (a+b)}{4(c^2+1)}x^2+\frac{c^2 (a+b)-c   \Delta }{c^2+1}xy,\\
\dot y=&\frac{3 a^2-8 a b c^2-2 a b+3 b^2}2x-2    (a+b)cy+\frac{\gamma}{4(c^2+1)}x^2\\
&+\frac{c(12 a^2 c^2+7 a^2-10 a b+12 b^2 c^2+7 b^2)-  (a+b)(6c^2+1)\Delta}{2(c^2+1)}xy+\frac{2 c^2 (a+b)-2 c \Delta }{c^2+1}y^2,
\end{split}
\]
where $\Delta= \sqrt{4c^2(a-b)^2+(3a-b)(a-3b)}$ and
\begin{multline*}
\gamma=(4 c^2+1) (a+b) (4 a^2 c^2+3 a^2-4 a b c^2-6 a b+4 b^2 c^2+3 b^2)\\-c \Delta (8 a^2   c^2+5 a^2+8 a b c^2+2 a b+8 b^2 c^2+5 b^2).
\end{multline*}
From the invariant algebraic curve of \eqref{YabQS} we obtain a (complex) invariant algebraic curve of degree $4$.

\smallskip

This illustrates that the Cremona transformation can be applied also on complex singular points. Here the eigenvalues of the singular points $(0:0:1)$ and $(0:1:0)$ allow to apply a Cremona quadratic transformation of type (C2) to obtain a new quadratic differential system. Since $4c^2(a-b)^2+(3a-b)(a-3b)<0$, the new differential system is complex, non-real.
\end{example}

Quadratic foliations whose singular points are all simple are in the cases {\rm(5)} and {\rm (7)} of Corollary \ref{T.Main.G}. For such foliations the following sharper characterization holds.

\begin{theorem}\label{T.Main.SSP}
Let  $\mathcal{F}$  be a complex projective foliation of degree $d(\mathcal{F})=2$ whose singular points are all simple. There exists a quadratic plane Cremona map $\Phi $ transforming  $\mathcal{F}$ into a quadratic foliation  $\Phi_{\ast}(\mathcal{F})$ if and only if among the proper singular points of $\mathcal{F}$ there is a node with integer eigenvalues $(a,1)$, with $0< a \leq 3$, such that any line through it is either transversal to the foliation or a first order tangent.
\end{theorem}

\begin{proof}
Under the hypothesis of the statement and according to Corollary \ref{T.Main.G}, the transformed projective foliation $\Phi_{\ast}(\mathcal{F})$ is quadratic if and only if
the three base points of $\Phi$ are in the cases {\rm(5)} or {\rm (7)} of Corollary \ref{T.Main.G}.
As noticed in Remark \ref{R.Seiden} the multiplicities of a foliation do not increase on further blowing-ups. Hence there can be no singular point of the foliation infinitely near to a regular one.
As regards the dicritical singular points, since they are simple by hypothesis, they can have no other dicritical singular point infinitely near to any of them.

This limits the possibilities for the base points of a quadratic plane Cremona map $\Phi$ to be one of those listed in Figure \ref{F.Combinacions}.
Observe that in any case there is some proper or infinitely near base point of $\Phi$ which is a star-node of the foliation $\mathcal{F}$. This gives the existence of the proper singular node with the restrictions of the statement.

Conversely, if the foliation $\mathcal{F}$ has a proper singular node satisfying the conditions of the statement, then after at most three blowing-ups we come to a star-node of the foliation. Then Proposition \ref{P.ExistCre} assures that there exists a plane Cremona map $\Phi$ by fixing its three base points as follows: take the infinitely near star-node and any point (proper or infinitely near) preceding it (which gives altogether at most three points), and complete to a trio by taking any other singular point of the foliation $\mathcal{F}$ (whose existence is guaranteed by B\'{e}zout's Theorem). Since we have constructed a Cremona map $\Phi$ which satisfies the conditions of case {\rm (7)} of Corollary \ref{T.Main.G}, we are done.
\end{proof}

\begin{remark}\label{R.descartant1}
Figure \ref{F.Combinacions} show the different coincidences that can occur, according to Theorem \ref{T.Main.SSP}, between the base points of the quadratic plane Cremona map $\Phi $ and the singular points of the quadratic foliation $\mathcal{F}$ when the degree of ${\Phi }_{\ast}\mathcal{F}$ remains invariant. Each figure represents the features of a class of pairs $(\mathcal{F}, \Phi ) $ modulus projectivity.
\end{remark}

\begin{figure}[h]
\centering
\includegraphics[width=5cm]{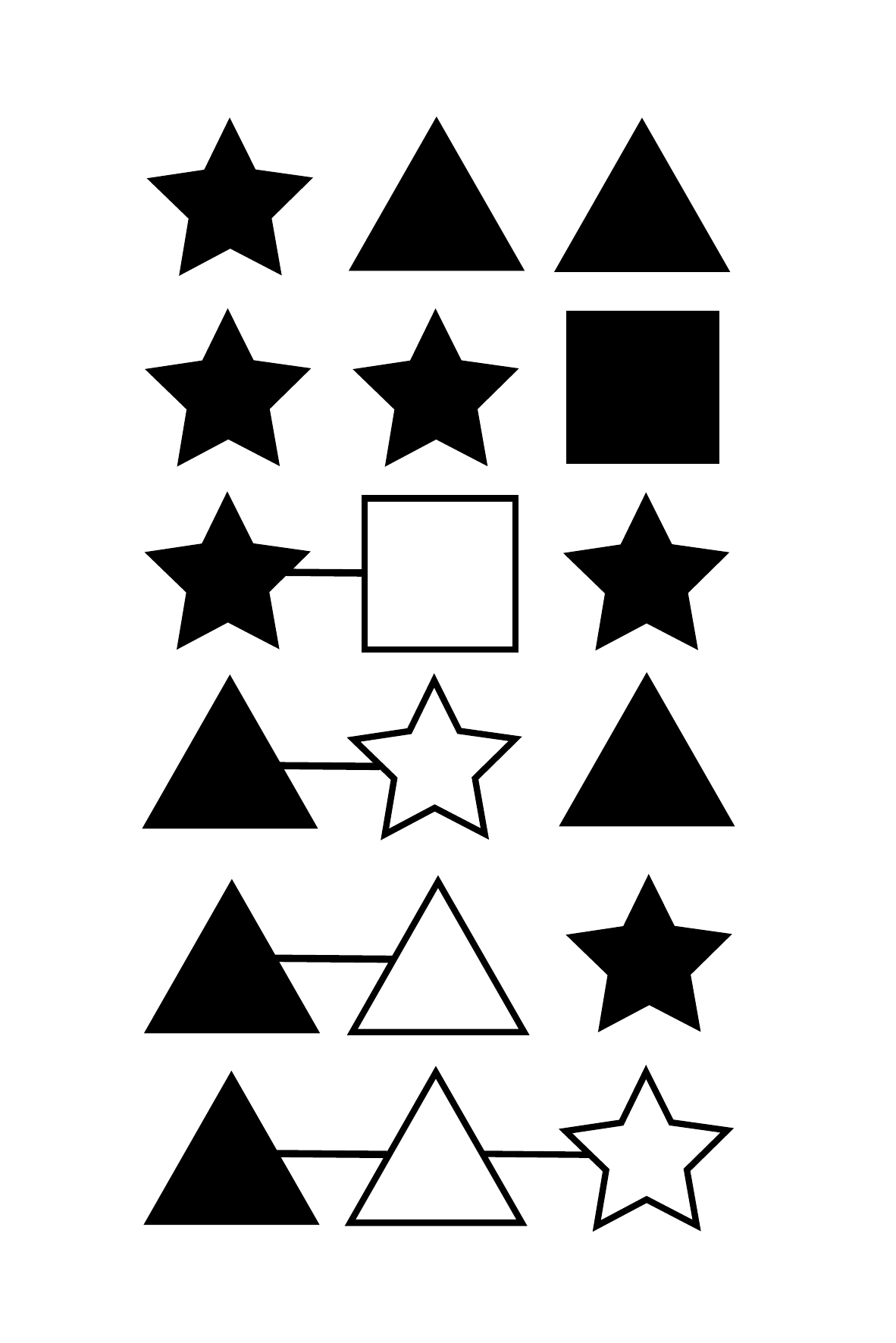}
\caption{\footnotesize The list of all possibilities for the base points of a quadratic plane Cremona map under the hypothesis of Theorem \ref{T.Main.SSP}. Proper points are represented black filled while infinitely near points are represented white filled (joined with an edge to the point which they are proximate to). The shapes relate the points to the foliation: square is a regular point, triangle is a non-dicritical singular point, star is a dicritical singular point.}\label{F.Combinacions}
\end{figure}

None of the known families of quadratic differential systems having an algebraic limit cycle has a multiple singular point, hence we shall only be concerned with cases {\rm(5)} and {\rm (7)} of Corollary \ref{T.Main.G} and the result of Theorem \ref{T.Main.SSP} will apply.

\begin{remark}\label{R.descartant2}
We note that none of the known families of quadratic differential systems having an algebraic limit cycle has a dicritical singular point, neither finite nor infinite. Hence  first, second, third and fifth cases in Figure \ref{F.Combinacions} are discarded.
%
%
We have checked the remaining forth and sixth configurations shown Figure \ref{F.Combinacions} and we have found out that the sixth case do not apply for the known families of quadratic differential systems having an algebraic limit cycle. Only the fourth case of Figure \ref{F.Combinacions} (which corresponds to the Cremona transformation (C2)) applies. Indeed it applies to Qin, Yablonski, CLS, CLS5 and CLS6 families.
\end{remark}

We provide in the next section the quadratic differential system that is obtained after applying the Cremona transformation (C2) to a quadratic differential system satisfying the hypothesis corresponding to the fourth case of Figure \ref{F.Combinacions}. From Remarks \ref{R.descartant1} and \ref{R.descartant2}, this is the only case in which we can apply Cremona transformations to the known families of quadratic differential systems having an algebraic limit cycle and obtain quadratic differential systems.


\section{Main results on limit cycles}\label{S.results}


Consider the quadratic differential system \eqref{e1}, that we write as
\begin{equation}\label{PQgen}
\begin{split}
\dot x&=p(x,y)=a_{00}+a_{10}x+a_{01}y+a_{20}x^2+a_{11}xy+a_{02}y^2,\\
\dot y&=q(x,y)=b_{00}+b_{10}x+b_{01}y+b_{20}x^2+b_{11}xy+b_{02}y^2.
\end{split}
\end{equation}

As we have seen in the previous section, only the situation described in the fourth case of Figure \ref{F.Combinacions} is useful for our purposes. In order to have the three base points of the Cremona transformation as in the fourth case of \ref{F.Combinacions}, we must take into account that:
\begin{enumerate}
\item The point $(0:0:1)$ is a singular point of system \eqref{PQgen} if and only if $a_{00}=b_{00}=0$.
\item The point $(0:1:0)$ is a singular point of system \eqref{PQgen} if and only if $a_{02}=0$.
\item In the previous case, the singular point infinitely near $(0:1:0)$ and on the direction of infinity is dicritical if and only if  $b_{02}=2a_{11}$.
\end{enumerate}

Now we can state our first main theorem.

\begin{theorem}\label{T.C2a}
Consider the differential system \eqref{PQgen}. If $a_{00}=b_{00}=0$, $a_{02}=0$ and $b_{02}=2a_{11}$,  then system \eqref{PQgen} is transformed, after the  quadratic Cremona transformation {\rm (C2)}, into the  quadratic differential system
\begin{equation}\label{PQC2a}
\dot x=x^3p\left(\frac 1x,\frac{y}{x^2}\right),\quad\dot y=2x^2y\,p\left(\frac 1x,\frac{y}{x^2}\right)-x^3q\left(\frac 1x,\frac{y}{x^2}\right).
\end{equation}
\end{theorem}

Theorem \ref{T.C2a} is proved in section \ref{P.TC2a}.

\begin{remark}\label{R.T.Cre}
The Yablonskii system \eqref{YabQS}  satisfies the hypotheses of  Theorem \ref{T.C2a}. Applying this theorem, we obtain  Qin system \eqref{QinQS} after the transformation. This was already shown in \cite{CLS}.

\smallskip

The {\rm CLS} system \eqref{CLS4QS}, after interchanging the variables $x$ and $y$ and moving any of the three singular points different from the focus to the origin, satisfies the conditions of  Theorem \ref{T.C2a}. So three different systems of type \eqref{PQC2a} may be obtained, one for each of the three finite singular points that were moved to the origin. Indeed, two of the new differential systems are the  {\rm CLS5} system \eqref{CLS5QS} and the {\rm CLS6} system \eqref{CLS6QS}, which arise in  this way from {\rm CLS} in \cite{CLS}. The third one is new and is presented in the following theorem.
\end{remark}

Applying the plane Cremona map (C2) to these systems above, we obtain different classes of birationally equivalent differential systems.
Section \ref{S.rel} explains the transformations among these systems with further detail.



\begin{theorem}\label{T.New5}
The quadratic differential system
\begin{equation}\label{EqNew5}
\begin{split}
\dot x&=-8x+\frac \gamma 2(\gamma-16) y- (5\gamma-64) x^2   +\frac\gamma 8 (\gamma^2-256) xy,\\
\dot y&= - 28 y+\frac{24}{\gamma} x^2-3 (3 \gamma-32) x y +\frac\gamma 4 (\gamma^2-256) y^2,
\end{split}
\end{equation}
has an irreducible invariant algebraic curve of degree five given by
\begin{equation}\label{fNew5}
\begin{split}
f(x,y)=&\, \gamma  y^2-4x^2y+\frac\gamma 2  (\gamma-12) x y^2 -\frac{\gamma^2}4 (\gamma-16)y^3   +  \frac{4}{\gamma}x^4\\
&-(\gamma-24)x^3 y+ \frac{\gamma}{16} (\gamma^2-256)x^2y^2 -\frac{24}{\gamma} x^5 +(\gamma+16) x^4y.
\end{split}
\end{equation}\normalsize
Its cofactor is
\[
k(x,y)=-56 -2(13\gamma-152) x +\frac{3\gamma}4 (\gamma^2-256) y.
\]

When $\gamma\in(0,8-3\sqrt{7})$, this algebraic curve contains an algebraic limit cycle of degree five. Indeed it has two components; one of them is an oval and the other is homeomorphic to a straight line. This last component contains two singular points of the system.

\smallskip

The phase portrait of system \eqref{EqNew5} is not topologically equivalent to the phase portrait of system {\rm CLS5}.
\end{theorem}

We call system \eqref{EqNew5} {\rm AFL5}. Theorem \ref{T.New5} is proved in section \ref{S:new5}. Figure \ref{F.New5} shows the phase portrait of system \eqref{EqNew5} on the Poincar\'e disk.

\begin{figure}[h!]
\centering
\includegraphics[width=6cm]{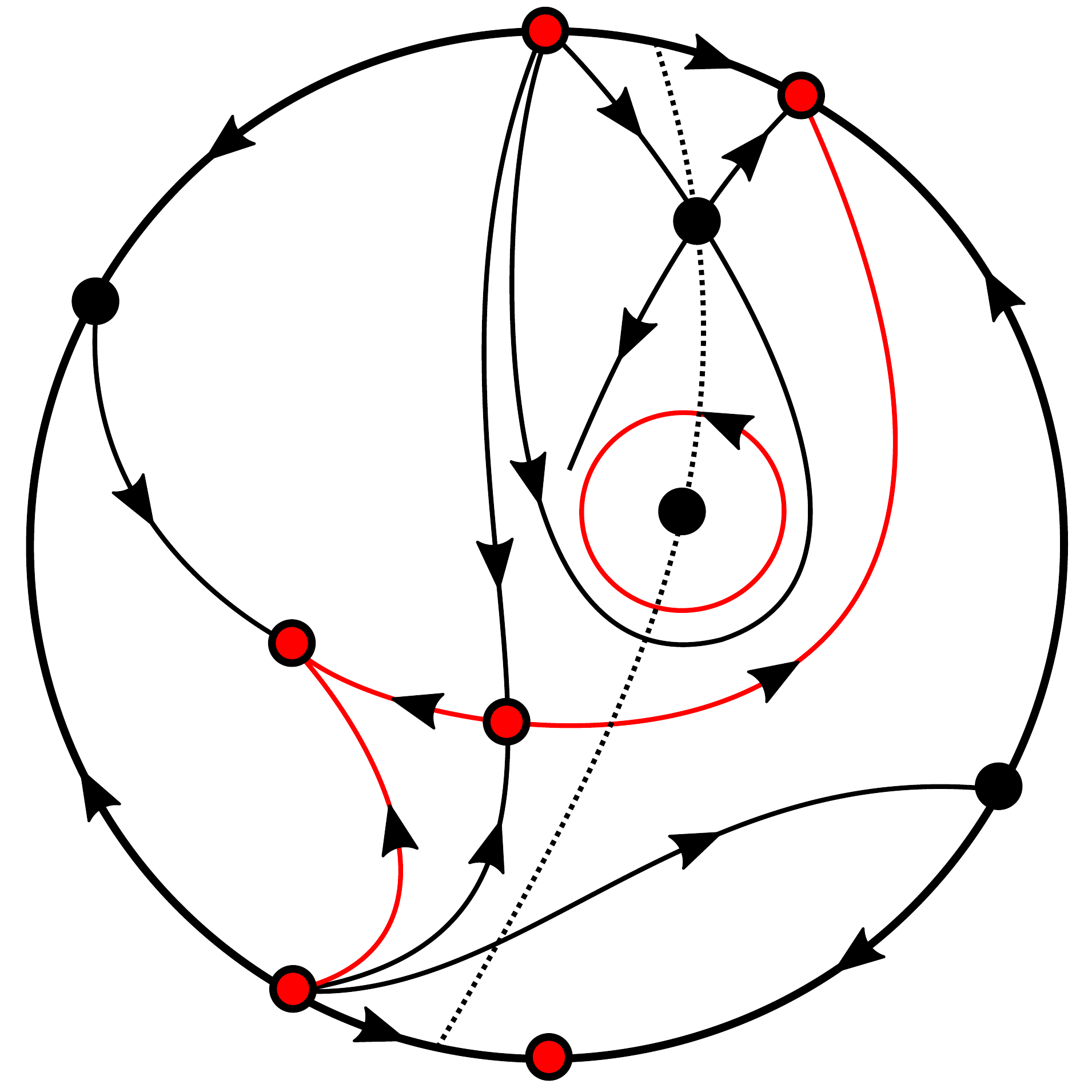}
\caption{\footnotesize Phase portrait of the quadratic differential system \eqref{EqNew5} having an algebraic limit cycle of degree 5 on the Poincar\'e disk. The red lines correspond to the invariant algebraic curve that contains the limit cycle. The dashed line corresponds to its cofactor.}\label{F.New5}
\end{figure}

\smallskip

After Remark \ref{R.T.Cre} and Theorem \ref{T.New5}, Figure \ref{F.EsquemaPP} shows all the non-equivalent phase portraits of all the known families of quadratic differential systems which have an algebraic limit cycle. The circled phase portraits correspond to those differential systems related by a quadratic Cremona transformation.

\begin{figure}[t]
\centering
\includegraphics[width=15cm]{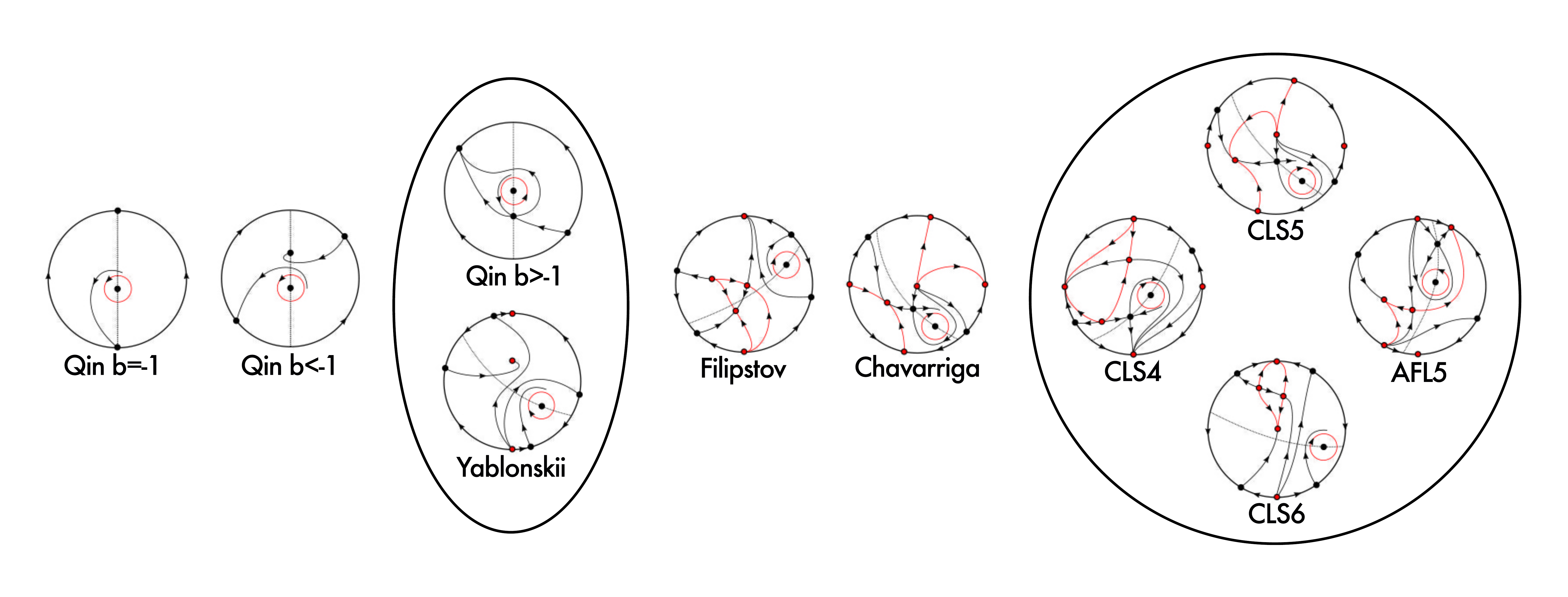}
\caption{\footnotesize Phase portraits of all the known families of quadratic differential systems having an algebraic limit cycle. The differential systems corresponding to the phase portraits inside  circles are related by the Cremona transformation \eqref{C2}. We note that we consider the different phase portraits of the Qin family depending on the value of $b$.}\label{F.EsquemaPP}
\end{figure}

\section{Proof of Theorem \ref{T.C2a}}\label{P.TC2a}

We know from section \ref{S.PVF} that the quadratic differential system \eqref{PQgen} can be thought in $\CP^2$ as a projective 1-form taking $P(X, Y, Z)=Z^2p(X/Z, Y/Z)$ and $Q(X, Y, Z)=Z^2q(X/Z, Y/Z)$. This 1-form writes as
\begin{equation}\label{Omega}
\Omega=AdX+BdY+CdZ,
\end{equation}
where we have defined $A=-ZQ$, $B=ZP$ and $C=XQ-YP$. The differential system \eqref{PQgen} is equivalent to the 1-form \eqref{Omega} of $\CP^2$.
Clearly it satisfies the Euler condition $XA+YB+ZC=0$.

\smallskip

Next lemma provides the expression of \eqref{Omega} after the application of the  Cremona transformation \eqref{C2i}.

\begin{lemma}\label{L.Cre}
After applying the Cremona transformation \eqref{C2i}, the 1-form \eqref{Omega} becomes
\begin{equation}\label{Eq.PhiOmega}
\phi^*\Omega=W(\hat Q-2V\hat P)dU+UW\hat PdV+U(V\hat P-\hat Q)dW,
\end{equation}
where $\hat P$ and $\hat Q$ are such that $\bar P=UW\hat P$ and $\bar Q=W\hat Q$.
\end{lemma}

\begin{proof}
Let $\bar A=-U^2\bar Q$, $\bar B=U^2\bar P$ and $\bar C=W(U\bar Q-V\bar P)$ are the transformations of $A,B,C$, respectively, with $(\bar P,\bar Q)(U,V,W)=(P,Q)(UW,VW,U^2)$ the transformation of $(P,Q)$ by the Cremona map. After applying  \eqref{C2i} to $\Omega$ and using the corresponding Euler condition, we obtain
\[
\begin{split}
\phi^*\Omega&=\bar A(WdU+UdW)+\bar B(VdW+WdV)+\bar C(2UdU)\\
&=(W\bar A+2U\bar C)dU+W\bar BdV+(U\bar A+V\bar B)dW\\
&=(-U^2W\bar Q+2UW(U\bar Q-V\bar P)dU+U^2W\bar PdV+(U^2V\bar P-U^3\bar Q)dW\\
&=UW(U\bar Q-2V\bar P)dU+U^2W\bar PdV+U^2(V\bar P-U\bar Q)dW.
\end{split}
\]
Notice that, since $a_{00}=0$, $b_{00}=0$ and $a_{02}=0$, we have $W|\bar P$, $W|\bar Q$ and $U|\bar P$, respectively. So there exist homogeneous polynomials  $\hat P$ and $\hat Q$ such that $\bar P=UW\hat P$ and $\bar Q=W\hat Q$. We can now simplify the expression of $\phi^*\Omega$:
\[
\phi^*\Omega=
U^2W^2(\hat Q-2V\hat P)dU+U^3W^2\hat PdV+U^3W(V\hat P-\hat Q)dW.
\]
The lemma follows after removing the common factor $U^2W$ from the expression of $\phi^*\Omega$.
\end{proof}

\begin{proof}[Proof of Theorem \ref{T.C2a}]
The  1-form  \eqref{Eq.PhiOmega} has degree four, which may correspond to a cubic affine differential system. Since we want to obtain a quadratic affine differential system, the elements in the 1-form \eqref{Eq.PhiOmega} must have a common factor. From the coefficient of $dV$, we note that  this common factor might be either $U$, or $W$, or a linear factor dividing $\hat P$. In this last case, this linear factor would also divide $\hat Q$, and thus system \eqref{PQgen} would have a common factor, which is a contradiction.  Therefore the common factor may be only either $U$ or $W$. We distinguish these two cases next.

\smallskip

Equation \eqref{Eq.PhiOmega} has $U$ as common factor if and only if $U|(\hat Q-2V\hat P)$. We  notice that $\hat P|_{U=0}=a_{11}VW$ and $\hat Q|_{U=0}=b_{02}V^2W$, hence $U$ is a common factor if and only if $b_{02}=2a_{11}$. Since this condition is satisfied by  the hypothesis of the theorem, the 1-form \eqref{Eq.PhiOmega} has degree 3. Indeed, we have
\[
\phi^*\Omega=WRdU+W\hat PdV-(\hat Q-V\hat P)dW,
\]
where $R$ is such that $UR=\hat Q-2V\hat P$.

\smallskip

We can obtain all the possible invariant straight lines to place at infinity from \eqref{rectainv}. Direct computations provide only $W=0$. So taking  $W=0$ as the line at infinity,  the 1-form \eqref{Eq.PhiOmega} provides the quadratic affine differential system $U'=\hat P,V'=-R$  in the variables $(U,V)$.

\smallskip

We recall that
\begin{multline*}
Z^2(p,q)(X/Z,Y/Z)=U^4(p,q)\left(\frac WU,\frac{VW}{U^2}\right)\\=(P,Q)(UW,VW,U^2)=(\bar P,\bar Q)(U,V,W)=(UW\hat P(U,V,W),W\hat Q(U,V,W)).
\end{multline*}
So in order to obtain the differential system \eqref{PQC2a},  we set $W=1$ and undo from the differential system $U'=\hat P,V'=-R$ the previous transformations to obtain:
\[
\begin{split}
U'&=\hat P=\frac{\bar P}{U}=\frac{U^4p}{U}={U^3}p\left(\frac 1U,\frac{V}{U^2}\right),\\
V'&=-R=\frac{2V\hat P-\hat Q}U=\frac 1U\left(2V\frac{\bar P}{U}-{\bar Q}\right)={2U^2Vp\left(\frac 1U,\frac{V}{U^2}\right)-U^3q\left(\frac 1U,\frac{V}{U^2}\right)}.
\end{split}
\]
Changing $(U,V)$ to the  usual notaion $(x,y)$ we have \eqref{PQC2a} and thus the theorem follows.
\end{proof}

\section{Cremona transformations among the known quadratic systems with algebraic limit cycles}\label{S.rel}

\subsection{From Qin to Yablonskii and back}

We apply the Cremona transformation to obtain Yablonskii family from Qin family. We have the following proposition.

\begin{proposition}\label{QinSaddle}
If Qin system is transformed into Yablonskii's system after the Cremona transformation \eqref{C2i}, then  $2 a^2 + (b+1) (2 b+1)=0$ and Qin's system has a finite saddle.
\end{proposition}

\begin{proof}
We note that Qin family has always a focus, which is surrounded by the limit cycle. It also has another singular point, which can be either a saddle, or a node, or a focus. Moreover the singular point surrounded by the limit cycle may change depending on the value of the parameters. At infinity, we have a unique real singular point, whose eigenvalues are $-(a^2+ (b+1)^2)/a), -( b+1)/a$.

\smallskip

We apply Theorem \ref{T.C2a} to Qin family. First we need to assure that the hypotheses of the theorem are satisfied:\small
\[
\begin{split}
\dot x=&\left(3c-(4 b+1) \sqrt{4( b+1)+c^2}\right)x- (b+1)\left(3c+\sqrt{4( b+1)+c^2}\right)y+4 b x^2+2 (b+1) xy,\\
\dot y=& \big(x - (b +1)y\big)\left( \frac{2( (b+2) c- b \sqrt{4( b+1)+c^2})}{b+1}+4 x -4 y\right).
\end{split}
\]\normalsize
The hypothesis $\mathrm{(H_{8})}$ writes $2 a^2 + (b+1) (2 b+1)=0$. So only if this condition holds we can obtain a quadratic system after the transformation \eqref{C2i}, which is of course Yablonskii's.

\smallskip

The equality $2 a^2 + (b+1) (2 b+1)=0$ provides the eigenvalues $ -( b+1)/(2a)$ and $ -( b+1)/a$ for the singular point at infinity. This means that it is to be a node. Hence the finite singular point not surrounded by the limit cycle must be a saddle because the sum of the indices of the singular points in the Poincar\'e disk must be one by the Poincar\'e-Hopf Theorem, see \cite{DLA} for more details.
\end{proof}

A transformation from  Yablonskii's family into Qin's family  was already provided in \cite{CLS2004}, although without mentioning the Cremona transformation. We note that for Yablonskii system the hypothesis $\mathrm{(H_{8})}$ holds and then Theorem \ref{T.C2a} applies. After the Cremona transformation, we  obtain the differential system
\[
\begin{split}
\dot x&=3 (a + b) c x+ 4 y-4 a b c x^2  - (a + b) xy,\\
\dot y&=  \frac 12 (3 (a^2+b^2) - 2 a b - 8 a b c^2) x -  2 (a + b) c y-a b (a + b) x^2 - 4 a b c xy - 2 (a + b) y^2.
\end{split}
\]
The algebraic curve of the Yablonskii system becomes the oval
\[
a b \left(x-\frac{a + b}{2 a b}\right)^2 + (c + y)^2-\frac{(a - b)^2}{4 a b}=0.
\]
Its cofactor is $-8 a b c x - 4 (a + b) y$.

\smallskip

Since Qin's system is the only one having an algebraic limit cycle of degree two, this family coming from Yablonskii's family is a subfamily of  Qin's.  Indeed the above system has a saddle at the origin, as we proved in Proposition \ref{QinSaddle}. This means in particular that Qin's family having a focus and an antisaddle cannot be brought to Yablonskii family.


\subsection{From CLS to CLS5}\label{SS.45}

The quadratic Cremona transformation \eqref{C2} is also used in \cite{CLS2004} to bring the fourth family of quadratic differential systems having an algebraic limit cycle of degree four into the first known family having an algebraic limit cycle of degree five. Before applying the transformation \eqref{C2}, the singular point $(-1/(\alpha+4),(\alpha-2)/2)$ is moved to the origin and afterwards the variables $x$ and $y$ are interchanged in order to bring the infinite singular point in the direction $y=0$ to the direction $x=0$.

\smallskip

The affine  quadratic differential system that we obtain from \eqref{Eq.PhiOmega} is
\[
\begin{split}
\dot x&= -8x+ 2(\alpha^2-16)y -2 (5\alpha-12) x^2+  (\alpha^2-16) (\alpha+12)x y,\\
\dot y&=  - 28 y+\frac{12}{\alpha+4} x^2 - 6 (3\alpha-4) x y +  2  (\alpha^2-16) (\alpha+12) y^2.
\end{split}
\]
Notice that this system is CLS5 after swapping $x$ and $y$ and changing the sign of the time. We also get the algebraic curve  of degree 5
\begin{multline*}
\frac{4}{\alpha+4} x^4 -\frac{24}{\alpha+4}x^5- (8 x^2 + 4 (\alpha-8) x^3 - 4 (\alpha+12) x^4) y \\
+ (4 (\alpha+4) + 4 (\alpha-2) (\alpha+4) x + (\alpha^2-16) (\alpha+12) x^2) y^2 -  4 (\alpha^2-16)(\alpha+4) y^3=0.
\end{multline*}
Its cofactor is $-56 - 4 (13\alpha-24) x + 6 (\alpha^2-16) (\alpha+12) y$.

\subsection{From CLS to CLS6}\label{SS.46}

The family CLS is also brought  in \cite{CLS2004}, after the Cremona transformation \eqref{C2}, into the first known family of quadratic systems having an algebraic limit cycle of degree six. In this case, first  the singular point $(1/(\beta+2),-(3\beta+8)/14)$ is moved to the origin and  again the variables are interchanged in order to bring the infinite singular point in the direction $y=0$ to the direction $x=0$. The affine  quadratic differential system that we obtain from \eqref{Eq.PhiOmega} is
\[
\begin{split}
\dot x&= -8 x + \frac 27 (\beta^2-4) y +\frac 27 (13\beta+30) x^2 -  \frac 3{49} (\beta-30) (\beta^2-4) x y,\\
\dot y&= -28 y - \frac{12}{\beta+2} x^2+\frac 27 (31 \beta + 78) x y - \frac 6{49} (\beta - 30) (\beta^2 - 4) y^2.
\end{split}
\]
Notice that this system becomes CLS6 after the affine change of variables and time
\[
(x,y)=\left(-\frac 1{14} (\beta-30) x, \frac{\beta}{49} (\beta-30)^2 (14 x+3(\beta^2 -4) y)\right),\quad \frac{dt}{ds}= -3\beta (\beta-30).
\]
Moreover the algebraic curve of degree 6
\[
\begin{split}
196 &x^4 + 56 ( 2 \beta+3 ) x^5 +  8 ( \beta+12 ) ( 2 \beta+3 ) x^6 \\
&+ (392 ( \beta+2 ) x^2 +     28 ( \beta^2-4 ) x^3 -     12 ( \beta^2-4 ) ( 2 \beta+3 ) x^4) y\\
& + (196 ( \beta +2)^2 -     28 ( \beta+2 )^2 ( 3 \beta+8 ) x + 9 ( \beta^2-4 )^2 x^2) y^2 -  28 ( \beta^2-4 ) ( \beta+2 )^2 y^3=0.
\end{split}
\]
is obtained. Its cofactor is $-56 +12 (13\beta+30) x/7 - 18 (\beta-30 ) (\beta^2-4) y/49$.

\subsection{From CLS5 to CLS6}

We see in this subsection that the quadratic Cremona transformation \eqref{C2} can also be used to bring the family CLS5 having an algebraic limit cycle of degree five into the family CLS6 having an algebraic limit cycle of degree six. Before applying the transformation, the singular point
\[
\left(\frac{2 ( 35 \alpha^2 - 54 \alpha -288)-2 (13 \alpha-24 )\beta}{( \alpha-6) ( \alpha^2-16) (\alpha+12)^2}, \frac{14}{6 - 7 \alpha - 3\beta}\right)
\]
is moved to the origin and the variables are  interchanged. Recall that $\beta=\sqrt{ 7 \alpha^2-108}$.
The affine  quadratic differential system that we obtain from \eqref{Eq.PhiOmega} is\footnotesize
\[
\begin{split}
\dot x=&\frac{-24 \sqrt{7}+ 10 \sqrt{108 + \beta^2}}{\sqrt{7}}x  - \frac 1{49} ( \beta^2-4) (84 + \sqrt{7} \sqrt{108 + \beta^2}) y\\
&+\frac{7704 - 156 \sqrt{7} \sqrt{108 + \beta^2} +
    2 \beta (348 + 39 \beta - 17 \sqrt{7} \sqrt{108 + \beta^2})}{(\beta-30) (\beta+12)} x^2 +\frac{    6 (\beta-30) (\beta^2-4) }{    7 (6 - 3 \beta - \sqrt{7} \sqrt{108 + \beta^2})}xy,\\
\dot y=&\frac{84}{28 +      \sqrt{7} \sqrt{108 + \beta^2}}x+\frac{-24 \sqrt{7}+ 2 \sqrt{108 + \beta^2}}{\sqrt{7}}y \\
    &-168\frac{4956 - 186 \sqrt{7} \sqrt{108 + \beta^2} +
     \beta (-196 + 35 \beta + 9 \sqrt{7} \sqrt{108 + \beta^2})}{(\beta^2-4) (6 - 3 \beta - \sqrt{7} \sqrt{108 + \beta^2})^2}x^2 \\
     &+ \frac{22968 - 492 \sqrt{7} \sqrt{108 + \beta^2} +        2 \beta (1032 + 141 \beta - 59 \sqrt{7} \sqrt{108 + \beta^2})}{(\beta-30) (\beta+12)}xy + \frac{12 ( \beta-30 ) (\beta^2-4)}{ 7 (6 - 3 \beta - \sqrt{7} \sqrt{108 + \beta^2})}y^2.
\end{split}
\]\normalsize
This system becomes CLS6 after an affine change of variables and time. 

\section{Proof of Theorem \ref{T.New5}}\label{S:new5}

System \eqref{CLS4QS} has three finite singular points different from the focus. Two of them were already used in \cite{CLS2004} (see also sections \ref{SS.45} and \ref{SS.46}) to obtain new families having algebraic limit cycles of degree 5 and 6. So we move the third singular point
$\big(1/(\alpha-4), - 1 - \alpha/2\big)$ to the origin and swap the variables $x$ and $y$. The system writes:
\[
\begin{split}
\dot x=&-8 x^2 - (\alpha-12 ) (\alpha^2-4) y +   2 (5\alpha+12)x + 2 (\alpha^2-4)x y,\\
\dot y=&  \frac{12}{\alpha-4}x+ 2 (\alpha+12) y + 12x y+ 4 (\alpha^2-4) y^2.
\end{split}
\]

Now the hypotheses of Theorem \ref{T.C2a} are satisfied. Hence the Cremona transformation (C2) provides the quadratic differential system \eqref{EqNew5} with invariant infinity, and  the algebraic curve  \eqref{fNew5} of degree 5, where we have set $\gamma=2(4-\sqrt{16-a})$ for simplicity of the results.

\smallskip

The phase portrait of system \eqref{EqNew5} is non-topologically equivalent to that of system CLS5 because there is no affine change of variables and scaling of the time that sends one to the other. 
The irreducibility of \eqref{fNew5} follows from the irreducibility of \eqref{CLS4QS}.

\smallskip

Since the curve \eqref{fNew5} contains an algebraic limit cycle for $a\in(0,1/4)$, one may easily check that the  oval of \eqref{fNew5} does not intersect the straight lines of the Cremona transformation (C2), so the theorem follows.\cqb

\section*{Acknowledgements}

M. Alberich-Carrami\~nana is also with the Barcelona Graduate School of Mathematics (BGSMath), and she is partially supported by Spanish Ministerio de Econom\'ia y Competitividad grant MTM2015-69135-P and by Generalitat de Catalunya 2017SGR-932 project.

A. Ferragut and J. Llibre are  partially supported by the MINECO grants MTM2016-77278-P and MTM2013-40998-P.

A. Ferragut is partially supported by the Universitat Jaume~I grant P1-1B2015-16.

J. Llibre is partially supported by an AGAUR grant number 2014SGR-568, and the grant FP7-PEOPLE-2012-IRSES 318999.

\end{document}